\documentclass[oneside,english]{amsart}
\usepackage[T1]{fontenc}
\usepackage[latin9]{inputenc}
\usepackage{amssymb}
\usepackage{amsthm}
\usepackage{color}
\usepackage{amscd}

\makeatletter
\providecommand{\tabularnewline}{\\}

\numberwithin{equation}{section}
\numberwithin{figure}{section}
  
  \@ifundefined{theoremstyle}{\usepackage{amsthm}}{}
  \theoremstyle{plain}
  \newtheorem{thm}{Theorem}[section]
  \theoremstyle{plain}
  \newtheorem{cor}[thm]{Corollary}
  \theoremstyle{plain}
  \newtheorem{lem}[thm]{Lemma}
  \theoremstyle{definition}
  
  \theoremstyle{remark}
  \newtheorem{rem}[thm]{Remark}
  \theoremstyle{remark}
  \newtheorem{notation}[thm]{Notation}
  \theoremstyle{plain}
  \newtheorem{prop}[thm]{Proposition}
 \theoremstyle{definition}
  \newtheorem{example}[thm]{Example}
  \theoremstyle{remark}
  \newtheorem*{acknowledgement*}{Acknowledgement}

\newcounter{casectr}
\newenvironment{caseenv}
{\begin{list}{{\itshape\ Case} \arabic{casectr}.}{%
 \setlength{\leftmargin}{\labelwidth}
 \addtolength{\leftmargin}{\parskip}
 \setlength{\itemindent}{\listparindent}
 \setlength{\itemsep}{\medskipamount}
 \setlength{\topsep}{\itemsep}}
 \setcounter{casectr}{0}
 \usecounter{casectr}}
{\end{list}}

\newcommand{\overseteqnum }[2]{\overset{\makebox{\tiny \eqref{#1}}}{#2}}

\newcommand{\Inn}[1]{\mathrm{Inn}(#1)}
\newcommand{\trid}{\triangleright}

\newcommand{\obl }{\lfloor }
\newcommand{\obr }{\rfloor }
\newcommand{\orbit}{\mathcal{O}}
\newcommand{\ndN }{\mathbb{N}_0}
\newcommand{\ndZ }{\mathbb{Z}}
\newcommand{\ndQ }{\mathbb{Q}}
\newcommand{\ndC }{\mathbb{C}}
\newcommand{\fie }{\Bbbk }
\newcommand{\NA}{\mathfrak{B}}
\newcommand{\lspan}[1][\fie]{\mathrm{span}_{#1}}

\makeatother

\usepackage{babel}

\begin{document}

\title{Nichols algebras of group type with many quadratic relations}
\author{M.~Gra{\~n}a}
\address{M.~Gra{\~n}a, Departamento de matem\'atica, FCEyN, Universidad de Buenos Aires, Pabell\'on 1, Ciudad Universitaria (1428), Buenos Aires, Argentina}
\email{matiasg@dm.uba.ar}

\author{I.~Heckenberger}
\address{I.~Heckenberger, Philipps-Universit\"at Marburg, FB Mathematik und
Informatik, Hans-Meerwein-Stra\ss e, 35032 Marburg, Germany}
\email{heckenberger@mathematik.uni-marburg.de}

\author{L.~Vendramin}
\address{L.~Vendramin, Departamento de matem\'atica, FCEyN, Universidad de Buenos Aires, Pabell\'on 1, Ciudad Universitaria (1428), Buenos Aires, Argentina}
\address{Instituto de Ciencias, Universidad de Gral. Sarmiento, J.M. Gutierrez 1150, Los
Polvorines (1653), Buenos Aires, Argentina}
\email{lvendramin@dm.uba.ar}

\begin{abstract}
  We classify Nichols algebras of irreducible Yetter-Drinfeld modules over
  nonabelian groups satisfying an inequality for the dimension
  of the homogeneous subspace of degree two.
  All such Nichols algebras are finite-dimensional,
  and all known finite-dimensional Nichols algebras of nonabelian group type
  appear in the result of our classification. We find a new finite-dimensional
  Nichols algebra over fields of characteristic two.
\end{abstract}

\maketitle

\section{Introduction}

In many mathematical theories the classification of fundamental objects is a
central and difficult problem. Successful classifications like the one of
finite Coxeter groups, finite simple groups, and finite-dimensional
simple Lie algebras serve as archetype for other theories. Often a finiteness
assumption is imposed to achieve a reasonable result.

The theory of Hopf algebras covers classes of examples like group algebras,
enveloping algebras of Lie algebras, regular functions on finite and
affine algebraic
groups, quantum groups and others. Additionally, many examples
admit various kinds of deformations. With the discovery of quantum groups by
Drinfeld and Jimbo \cite{MR934283, MR841713} it became clear that non-commutative and
non-cocommutative Hopf algebras are natural objects playing an important role
in many mathematical and physical theories. The large variety of examples
forces any classification ansatz to restrict oneself to a special class.

A Hopf algebra is \textit{pointed} if all irreducible subcomodules are
one-dimensional. Examples of pointed Hopf algebras are
universal enveloping algebras of Lie algebras and restricted
Lie algebras and quantized enveloping algebras. The first non-commutative
non-cocommutative examples have been the Sweedler Hopf algebra \cite{MR0252485}
and the Taft Hopf algebra \cite{MR0286868}. A systematic study of pointed Hopf algebras was
started by Nichols \cite{MR0506406}. The classification enjoyed significant progress
with the invention of the lifting method by Andruskiewitsch and Schneider \cite{MR1659895}.
The idea of the method is to understand first the coradically graded Hopf
algebras and then to lift them to non-graded Hopf algebras.
Coradically graded pointed Hopf algebras are biproducts (bosonizations) of a group
algebra and a graded braided Hopf algebra with primitives concentrated in
degree one.
The latter is a \textit{Nichols
algebra} if it is generated by the elements of degree one.
The subspace of primitives is a \textit{Yetter-Drinfeld module} over the group
algebra
and determines the Nichols algebra uniquely.
Based on the lifting method, Andruskiewitsch and Schneider \cite{as-ann} classified
finite-dimensional pointed Hopf algebras with abelian coradical
under the assumption that the order of any group-like element is
relatively prime to $2$, $3$, $5$, and $7$.
The proof of this remarkable result makes heavily use of the classification of
Nichols algebras of diagonal type \cite{MR2207786, MR2462836}.

The next natural step in the classification of pointed Hopf algebras is to ask
for finite-dimensional Nichols algebras over nonabelian groups.
This problem turns out to be extremely difficult. So far only a small number
of examples appeared in the literature, and there is no obvious way to
describe them in a unified way. Fomin and Kirillov \cite{MR1667680}
studied Nichols algebras over symmetric groups and they used them to
analyze the geometry of Schubert cells. These and other examples have been
found independently by Milinski and Schneider \cite{MR1800714}. Andruskiewitsch
and the first author investigated Nichols algebras in terms of racks and
detected yet another examples \cite{MR1994219}. At the moment, the calculation of the
Hilbert series of these algebras requires the use of computer algebra.
In general,
the lack of knowledge about the general structure of Nichols algebras
does not allow to decide if a given Nichols algebra is
finite-dimensional.

In this paper we introduce a condition on the Nichols algebra in form of an
inequality concerning the dimension of the degree two subspace.
Our goal is to classify Nichols algebras of irreducible Yetter-Drinfeld
modules over nonabelian groups under this assumption. Our main result is Theorem \ref{thm:YD} 
which states that all such Nichols algebras are finite-dimensional, and all known examples
appear this way. We recall the list of examples in Table \ref{tab:nichols}. 
Our approach is very general and does not need any assumption
on the base field. It relies heavily on our Theorem \ref{thm:racks} which is a classification of racks 
satisfying certain inequality. With Theorem \ref{thm:diff_char} we present a method 
to compare Nichols algebras defined over
different fields. By consequent application of our theory it was possible to
find a substantially new example of a finite-dimensional
Nichols algebra over fields of
characteristic $2$, see Proposition \ref{pro:Tchar2}.

\begin{table}[ht]
\caption{Finite-dimensional Nichols algebras}
\begin{center}
\label{tab:nichols}
\begin{tabular}{|r|r|l|l|}
\hline 
Rank & Dimension & Hilbert series & Remark \tabularnewline
\hline 
3 & $12$ & $(2)^2_t (3)_t$ & Prop. \ref{pro:D3_main}\\
\hline
4 & $36$ & $(2)^2_t (3)^2_t$ & Prop. \ref{pro:T_main}(3), $\mathrm{char}\,\fie=2$\\
\hline
4 & $72$ & $(2)^2_t (3)_t (6)_{t}$ & Prop. \ref{pro:T_main}(2), $\mathrm{char}\,\fie\ne2$ \\
\hline
5 & $1280$ & $(4)^4_t (5)_t$ & Prop. \ref{pro:aff_main}(2) \\
\hline
5 & $1280$ & $(4)^4_t (5)_t$ & Prop. \ref{pro:aff_main}(2)\\
\hline
6 & $576$ & $(2)^2_t (3)^2_t (4)^2_t$ & Prop. \ref{pro:A_main} \\
\hline
6 & $576$ & $(2)^2_t (3)^2_t (4)^2_t$ & Prop. \ref{pro:B_main} \\
\hline
7 & $326592$ & $(6)^6_{t} (7)_t$ & Prop. \ref{pro:aff_main}(3)\\
\hline
7 & $326592$ & $(6)^6_{t} (7)_t$ & Prop. \ref{pro:aff_main}(3)\\
\hline
10 & $8294400$ & $(4)^4_t (5)^2_t (6)^4_t$ & Prop. \ref{pro:C_main}\\
\hline
10 & $8294400$ & $(4)^4_t (5)^2_t (6)^4_t$ & Prop. \ref{pro:C_main}\\
\hline
\end{tabular}
\par\end{center}
\end{table}

\section{Racks}

\subsection{Generalities}

A \textit{rack} is a pair $(X,\trid)$, where $X$ is a non-empty
set and $\trid:X\times X\to X$ is a map (considered as a binary operation on
$X$) such that
\begin{enumerate}
  \item[(R1)] \label{en:rack1}
  the map $\varphi_i:X\to X$, where $x\mapsto i\trid x$, is bijective
  for all $i\in X$, and
  \item[(R2)] \label{en:rack2}
    $i\trid(j\trid k)=(i\trid j)\trid(i\trid k)$
    for all $i,j,k\in X$.
\end{enumerate}
A rack $(X,\trid )$, or shortly $X$,
is a \textit{quandle} if $i\trid i=i$ for all $i\in X$.

The \textit{inner group} of a rack $X$ is the group generated by 
the permutations
$\varphi_i$ of $X$, where $i\in X$. We write $\Inn X$
for the inner group of $X$.
Axiom~(R2) implies that 
\begin{gather}
\label{eq:conjugation}
\varphi_{i\trid j}=\varphi_i\varphi_j\varphi_i^{-1}
\end{gather}
for all $i,j\in X$.
More generally, the following holds.
\begin{lem}
\label{lem:conjugation}
Let $X$ be a rack and
let $k\in\ndN $ and $i_1,i_2,\dots ,i_k,j,l\in X$
such that $\varphi_{i_1}\varphi_{i_2}\cdots\varphi_{i_k}(j)=l$.
Then $\varphi_{i_1}\varphi_{i_2}\cdots\varphi_{i_k}\varphi_j\varphi_{i_k}^{-1}\cdots\varphi_{i_2}^{-1}\varphi_{i_1}^{-1}=\varphi_l$.
\end{lem}

\begin{proof}
By induction on $k$ using Equation \eqref{eq:conjugation}.
\end{proof}

A \textit{subrack} of a rack
$X$ is a non-empty subset $Y\subseteq X$ such
that $(Y,\trid)$ is also a rack.
We say that a rack $X$ is \textit{indecomposable} if the inner group $\Inn X$
acts transitively on $X$. Also, $X$ is \textit{decomposable} if it is not
indecomposable. 

Let $(X,\trid)$ and $(Y,\trid)$ be racks. A map
$f:(X,\trid)\to(Y,\trid)$ is a \textit{morphism} of racks if
$f(i\trid j)=f(i)\trid f(j)$, for all $i,j\in X$. 

\begin{example}
\label{exa:racks}
A group $G$ is a rack with $x\trid y=xyx^{-1}$ for all $x,y\in G$. 
If a subset
$X\subseteq G$ is stable under conjugation by $G$, then it is a subrack of $G$.
In particular, we list the following examples.
\begin{enumerate}
\item The rack given by the conjugacy class of involutions in $G=\mathbb{D}_p$, the dihedral
group with $2p$ elements, has $p$ elements. It is called the \textit{dihedral rack} (of order $p$) 
and will be denoted by $\mathbb{D}_p$.
\item Let $\mathcal{T}$ be the conjugacy class of $(2\,3\,4)$
in $\mathbb{A}_4$.
We use the following labeling for the elements of $\mathcal{T}$:
$\pi_1=(2\,3\,4)$, $\pi_2=(1\,4\,3)$, $\pi_3=(1\,2\,4)$, $\pi_4=(1\,3\,2)$.
This is the rack associated with the vertices of the tetrahedron. 
\item Let $\mathcal{A}$ be the conjugacy class of $(1\,2)$ in $\mathbb{S}_4$.
We use the following labeling for the elements of $\mathcal{A}$:
$\pi_1=(3\,4)$, $\pi_2=(2\,3)$, $\pi_3=(2\,4)$, $\pi_4=(1\,2)$, $\pi_5=(1\,3)$, $\pi_6=(1\,4)$.

\item Let $\mathcal{B}$ be the rack given by the permutations 
$\varphi_1=(2\,3\,4\,5)$, $\varphi_2=(1\,5\,6\,3)$, $\varphi_3=(1\,2\,6\,4)$, $\varphi_4=(1\,3\,6\,5)$, $\varphi_5=(1\,4\,6\,2)$, $\varphi_6=(2\,5\,4\,3)$. This rack can be realized as the conjugacy class of $4$-cycles in $\mathbb{S}_4$. 
The rack $\mathcal{B}$ is not isomorphic to $\mathcal{A}$. Indeed, for all $i\in\mathcal{A}$
the map $\varphi_i\in\Inn{\mathcal{A}}$ is an involution,
but $1\trid(1\trid2)=4$ in $\mathcal{B}$.

\item Let $\mathcal{C}$ be the conjugacy class of $(1\,2)$ in $\mathbb{S}_5$.
We use the following labeling for the elements of $\mathcal{C}$:
$\pi_1=(1\,2)$, $\pi_2=(2\,3)$, $\pi_3=(1\,3)$, $\pi_4=(2\,4)$, $\pi_5=(1\,4)$, $\pi_6=(2\,5)$, 
$\pi_7=(1\,5)$, $\pi_8=(3\,4)$, $\pi_9=(3\,5)$, $\pi_{10}=(4\,5)$.
\end{enumerate}
\end{example}

\begin{example}
\label{exa:affine}
Let $p$ be a prime number and let $q$ be a power of $p$. Let $X=\mathbb{F}_q$,
the finite field of $q$ elements.  For $0\ne\alpha\in\mathbb{F}_q$ we have a
rack structure on $X$ given by $x\trid y=(1-\alpha)x+\alpha y$ for all
$x,y\in X$.  This rack is called the \textit{affine rack} associated to the
pair $(\mathbb{F}_q,\alpha)$ and will be denoted by $\mathrm{Aff}(q,\alpha)$.
These racks are also called Alexander quandles, see \cite{MR1990571}.
\end{example}

%

\begin{rem}
\label{rem:inner}
Let $X$ be a finite rack and assume that
$\Inn X$ acts transitively on $X$. Then for all $i,j\in
X$ there exist $r\in \ndN $ and
$k_1,k_2,\dots ,k_r\in X$ such that
$\varphi^{\pm1}_{k_1}\varphi^{\pm1}_{k_2}\cdots\varphi^{\pm1}_{k_r}(i)=j$.
Lemma~\ref{lem:conjugation} implies that
all permutations $\varphi_i$, where $i\in X$, have the
same cycle structure.
\end{rem}

\begin{lem}
\label{lem:rack_indecomposable}
Let $X$ be a finite rack, and let
$Y$ be a non-empty proper subset of $X$. The following are equivalent.
\begin{enumerate}
\item $X=Y\cup(X\backslash Y)$ is a decomposition of $X$.
\item $X\trid Y\subseteq Y$. 
\end{enumerate}
\end{lem}
\begin{proof}
See \cite[Lemma 1.14]{MR1994219}.
\end{proof}


\begin{lem}
\label{lem:racks_commute}
Let $X$ be a rack and
let $a,x,y,z\in X$ such that $x\trid y=z$, $a\trid x=x$ and $a\trid z=z$.
Then $a\trid y=y$.
\end{lem}

\begin{proof}
We have $x\trid y=z=a\trid z=a\trid(x\trid y)=(a\trid x)\trid(a\trid y)=x\trid (a\trid y)$. Thus the claim 
follows from (R1).
\end{proof}

\begin{lem}
\label{lem:racks_abcd}
 Let $X$ be a rack and let $a,b,c,d\in X$ such that $a\trid a=a$.
 If $a\trid c=b\trid c$, $a\trid d=b\trid d$ and $c\trid d=a$ then
 $b\trid a=a$.
\end{lem}

\begin{proof} $b\trid a=b\trid(c\trid
	d)=(b\trid c)\trid(b\trid d)
	=(a\trid c)\trid(a\trid
	d)=a\trid(c\trid d)=a\trid a=a$.
\end{proof}

We say that a rack is \textit{involutive} if $\varphi_i$ is an involution for
every $i\in X$, that is, if
$i\trid(i\trid j)=j$ for every $i,j\in X$. 
The dihedral racks and the racks $\mathcal{A}$ and $\mathcal{C}$ 
of Example \ref{exa:racks} are involutive.

\begin{lem}
\label{lem:involutions_abc}
Let $X$ be an involutive rack. Let $a,b,c\in X$ such that $a\trid a=a$.
If $b\trid a=c\trid a$ and $a\trid b=c\trid b$ then $a\trid(b\trid a)=b\trid a$.
\end{lem}

\begin{proof}
$a\trid(b\trid a)=(a\trid b)\trid(a\trid a)=(c\trid b)\trid a=c\trid (b\trid (c\trid a))=c\trid a=b\trid a$.
\end{proof}

\subsection{The enveloping group of a rack}
\label{ss:envgroup}

For any rack $X$,
the \textit{enveloping group} of $X$ is 
\[
G_X=F(X)/\langle iji^{-1}=i\trid j,\; i,j\in X\rangle,
\]
where $F(X)$ denotes the free group generated by $X$. 
We write $Z(G_X)$ for the center of $G_X$. 
For all $i\in X$ let $x_i$ be the image of
$i$ under $X \hookrightarrow F(X) \rightarrow G_X$.

\begin{rem}
\label{rem:grading}
The enveloping group $G_X$ of a rack $X$
is $\ndZ$-graded by $\deg{x_i}=1$ for all $i\in X$.
\end{rem}

A rack $X$ is \textit{faithful} if the
map $X\to \Inn X$ defined by $i\mapsto \varphi _i$ is injective,
see \cite[Def.\,1.11]{MR1994219}.
A rack $X$ is \textit{injective} if the map $X\to G_X$ defined by
$i\mapsto x_i$ is injective, see \cite[Def. 8]{MR1809284}.

\begin{lem}
  Any faithful rack is injective.
  \label{lem:faithful=>injective}
\end{lem}

\begin{proof}
  Let $X$ be a faithful rack and let $i,j\in X$.
  By (R1), (R2) and the definition of $G_X$ the group $G_X$ acts on $X$
  such that $x_k\trid l=k\trid l$ for all $k,l\in X$.
  If $i\not=j$ then $\varphi _i\not=\varphi _j$ since $X$ is faithful.
  Hence there exists $k\in X$ such that $x_i\trid k=i\trid k\not=j\trid
  k=x_j\trid k$. It follows that $x_i\not=x_j$ in $G_X$ and hence $X$ is
  injective.
\end{proof}

\begin{example} \label{exa:notinjective}
  We give a rack which is not injective.
  Let $X=\{1,2,3\}$ and $\trid :X\times X\to X$ the map with
  $1\trid i=2\trid i=i$ for all $i\in X$ and $3\trid 1=2$, $3\trid 2=1$,
  $3\trid 3=3$. Then $X$ is a rack (and even a quandle). In $G_X$ the
  relations $x_1x_3x_1^{-1}=x_3$ and $x_3x_1x_3^{-1}=x_2$ hold
  and hence $x_1x_3=x_3x_1=x_2x_3$. Thus $x_1=x_2$ in $G_X$,
  that is, $X$ is not injective.
\end{example}


\begin{example} \label{exa:notfaithful}
  We give a rack which is injective but not faithful.
  Let $X=\{1,2\}$ and $\trid :X\times X\to X$ the map with
  $i\trid j=j$ for all $i,j\in X$. 
  Then $\mathrm{Inn}(X)$ is trivial and $G_X\cong \ndZ ^2$.
  We conclude that $X$ is injective but not faithful. 
\end{example}

\begin{lem}
  Let $X$ be a rack and let
  $k\in\ndN$, $i_1,\dots,i_k,l\in X$, and $\epsilon _1,\dots,\epsilon _k\in
  \{1,-1\}$. 
  If $l$ is a fixed point of the permutation
  $\varphi_{i_1}^{\epsilon _1}\varphi_{i_2}^{\epsilon _2}\cdots
  \varphi_{i_k}^{\epsilon _k}$
  then  $x_{i_1}^{\epsilon _1}x_{i_2}^{\epsilon _2}\cdots x_{i_k}^{\epsilon
  _k}$ belongs to the centralizer of
  $x_l$ in $G_X$. 
  The converse is true if $X$ is injective.
\end{lem}

\begin{proof}
Since $x_i x_j=x_{i\trid j}x_i$ for all $i,j\in X$, we have
\begin{gather}
\label{eq:formula}
(x_{i_1}x_{i_2}\cdots x_{i_k}) x_l = x_{i_1\trid(i_2\trid\cdots\trid(i_k\trid l))}(x_{i_1}x_{i_2}\cdots x_{i_k})
\end{gather}
and the result follows if $\epsilon _i=1$ for all $i$. The general case is
similar.
\end{proof}

\begin{cor}
  Let $X$ be a rack and let
  $k\in\ndN$, $i_1,\dots,i_k,l\in X$, and $\epsilon _1,\dots,\epsilon _k\in
  \{1,-1\}$. 
  If
  $\varphi_{i_1}^{\epsilon _1}\varphi_{i_2}^{\epsilon _2}\cdots
  \varphi_{i_k}^{\epsilon _k}=\mathrm{id}$ then
  $x_{i_1}^{\epsilon _1}x_{i_2}^{\epsilon _2}\cdots x_{i_k}^{\epsilon _k}$
  is central in $G_X$. The converse is true if $X$ is injective.
\end{cor}

The following result is a special case of \cite[Lemma\,1.9(2)]{MR1994219}.

\begin{cor}
\label{cor:Inn(X)}
  Let $X$ be an injective rack. Then $\Inn X\simeq G_X/Z(G_X)$.
\end{cor}


Corollary~\ref{cor:Inn(X)} fails without the assumption that $X$ is injective.
Indeed, in Example~\ref{exa:notinjective} the group $G_X$ is abelian but
$\Inn X$ is not the trivial group.

For the rest of Subsection~\ref{ss:envgroup} let $X$ be a finite rack.
For all $i,j\in X$ and $n\in\mathbb{N}$ we define $i\trid^n j=\varphi_i^n(j)$.

\begin{lem}
\label{lem:formulas_GX}
Let $n\in\ndN$ and let $i,j\in X$. 
\begin{enumerate}
\item \label{left_n}$\displaystyle{x_i^n x_j=x_{i\trid^n j}x_i^n}$.
\item \label{right_n}$\displaystyle{x_i x_j^n=x_{i\trid j}^n x_i}$.
\end{enumerate}
\end{lem}

\begin{proof}
By induction on $n$.
\end{proof}

\begin{lem}
\label{lem:central}
Let $i\in X$ and let $n$ be the order of $\varphi_i$. Then $x_i^n\in Z(G_X)$.
\end{lem}

\begin{proof}
Follows from Lemma \ref{lem:formulas_GX}\eqref{left_n}.
\end{proof}

\begin{lem}
\label{lem:powers}
Assume that $X$ is indecomposable. Then all permutations $\varphi_i$, where $i\in X$, have the same order $n$.
Moreover, $x_i^n=x_j^n$ for all $i,j\in X$.
\end{lem}

\begin{proof}
Since $X$ is indecomposable, the first claim follows from Equation \eqref{eq:conjugation}. 
The second claim holds by Lemmas \ref{lem:formulas_GX}\eqref{right_n} and \ref{lem:central}.
\end{proof}

For the rest of Subsection~\ref{ss:envgroup}
assume that $X$ is (finite and)
indecomposable. Let $d$ denote the number of elements of $X$. 
For convenience we identify $X$ with $\{1,2,\dots,d\}$. 
Let $n$ be the order of $\varphi_1$. 
Lemma \ref{lem:central} 
implies that the subgroup $\langle x_1^n\rangle$ of $G_X$ generated by $x_1^n$
is normal in $G_X$. Let
\[
\overline{G_X}=G_X/\langle x_1^n\rangle
\]
and let $\pi:G_X\to\overline{G_X}$ be the canonical projection.

\begin{lem}
The group $\overline{G_X}$ is finite.
\end{lem}

\begin{proof}
First we prove that for all $i,j\in X$ with $i\ne j$ there exist $k,l\in X$ such that 
$x_kx_l=x_ix_j$, $k<l$ and $(k,l)$ is lexicographically not bigger than $(i,j)$. 
If $i<j$ then the claim is trivial. Assume that $i>j$. Then
$x_ix_j=x_jx_k$, where $k\in X$ with $j\trid k=i$. Since $(j,k)$ is lexicographically 
smaller than $(i,j)$, the claim follows by induction on $i$. 

By Lemmas \ref{lem:central} and \ref{lem:powers} all elements of $G_X$ take the form 
$x_{i_1}\cdots x_{i_k}x_1^{tn}$ for some $k\in\ndN$, $i_1,\dots,i_k\in X$ and $t\in\ndZ$.
By the first paragraph of the proof we may assume that $i_1\leq i_2\leq\cdots\leq i_k$ and 
by Lemmas \ref{lem:central} and \ref{lem:powers} we may assume that 
at most $n-1$ consecutive $x_{i_j}$ are equal. 
This implies the claim.
\end{proof}



\begin{lem}
\label{lem:centralizer}
The following hold.
\begin{enumerate}
\item\label{lem:pi}$C_{G_X}(x_1)=\pi^{-1}C_{\overline{G_X}}(\pi x_1)$.
\item\label{lem:centralizer_ab} If $C_{\overline{G_X}}(\pi x_1)$ is abelian then $C_{G_X}(x_1)$ is abelian.
\item\label{lem:centralizer_cy} If $C_{\overline{G_X}}(\pi x_1)$ is cyclic and generated by $\pi x_1$ then $C_{G_X}(x_1)$ is cyclic and generated by $x_1$.
\end{enumerate}
\end{lem}

\begin{proof}
$(1)$ The inclusion $\subseteq$ is trivial.
Let $\overline{z}\in C_{\overline{G_X}}(\pi x_1)$ and let $z\in G_X$ such that $\pi z=\overline{z}$. 
Then $\pi z=\pi(x_1 z x_1^{-1})$. This implies that $x_1 z x_1^{-1}=z x_1^{nm}$ for some $m\in\mathbb{Z}$.
Since $G_X$ is $\ndZ$-graded, see Remark \ref{rem:grading}, it follows that $m=0$. Hence $z\in C_{G_X}(x_1)$.

$(2)$ Let $x,y\in C_{G_X}(x_1)$. Then $\pi(xyx^{-1})=\pi y$ since $C_{\overline{G_X}}(\pi x_1)$ is abelian. Hence
$xyx^{-1}=yx_1^{nm}$ for some $m\in\ndZ$. Moreover $m=0$ since $G_X$ is $\ndZ$-graded. Thus $(2)$ holds.

$(3)$ Let $y\in C_{G_X}(x_1)$. Then $\pi y\in C_{\overline{G_X}}(\pi x_1)$, and hence $\pi y=(\pi x_1)^p$ for some 
$p\in\ndZ$ by assumption. Therefore $y=x_1^{p+nm}$ for some $m\in\ndZ$.
\end{proof}

\begin{rem}
\label{rem:centralizer}
Lemma \ref{lem:centralizer}\eqref{lem:pi} is a particular case of a more general fact. Let 
$p:H\to G$ be an epimorphism of groups and let $h\in H$ such that the restriction of 
$p$ to the conjugacy class of $h$ in $H$ is injective. Then $p^{-1}C_G(p h)=C_H(h)$. 
Indeed, the inclusion $\supseteq$ is trivial. Moreover, if $g\in H$ 
such that $(p g)(p h)(p g)^{-1}=(p h)$ then $ghg^{-1}=h$ by the assumption on $p$.
\end{rem}

\subsection{Braided action}
\label{ss:braided}

Let $G$ be a group. Let $X$ be a conjugacy
class of $G$, viewed as a rack with $x\trid y=xyx^{-1}$ for all $x,y\in X$. 
Assume that
$G$ is generated by $X$. 
Then $X$ is
an indecomposable rack by Lemma~\ref{lem:rack_indecomposable}. 
Moreover $X$ is injective by definition.
Let $d=\#X$. 
In what follows we assume that $X=\{x_1,x_2,\dots ,x_d\}$
is finite and that $d>1$.
When considering $X$ as a rack, we write $i$ for
$x_{i}$.

\begin{rem}
\label{rem:crossed_set}
The rack $X$ satisfies the following.
\begin{enumerate}
\item\label{quandle}$i\trid i=i$ for all $i\in X$.
\item\label{cs}For all $i,j\in X$ we have $i\trid j=j$ if and only if $j\trid i=i$.
\end{enumerate}
A rack satisfying properties \eqref{quandle} and \eqref{cs} is called a \textit{crossed set}.
\end{rem}

The following lemma will be used as an indecomposability criterion.

\begin{lem}
  \label{lem:cs_indecomposable}
  Let $Y$ be a finite crossed set which is indecomposable as a rack.
  Let $y\in Y$, and let
  $n$ be the number of elements of $Y$ not fixed by $\varphi _y$.
  Let $Y_1,Y_2\subseteq Y$ such that $Y_1\not=\emptyset $
  and that the following hold.
  \begin{enumerate}
    \item $y_1\trid z=z$ for all $y_1\in Y_1$, $z\in Y\setminus Y_2$.
    \item For all $y_2\in Y_2$ there exist $n$ elements
      of $Y_1$ not commuting with $y_2$.
  \end{enumerate}
  Then $Y_1=Y_2=Y$.
\end{lem}

\begin{proof}
  Since $Y$ is indecomposable, the number $n$ is independent of the choice of
  $y\in Y$.
  Since $Y$ is a crossed set, Remark~\ref{rem:crossed_set}\eqref{cs} and the
  definition of $n$
  imply that $y_2\trid Y_1=Y_1$ for all $y_2\in Y_2$.
  Further, $z\trid Y_1=Y_1$ for all $z\in Y\setminus Y_2$ by (1)
  and by Remark~\ref{rem:crossed_set}\eqref{cs}.
  Hence $Y=Y_1$ by Lemma~\ref{lem:rack_indecomposable} and since
  $Y_1\not=\emptyset $.
  Since $Y$ is indecomposable, assumption (1) and
  Lemma~\ref{lem:rack_indecomposable} imply that $Y_2=Y$.
\end{proof}

Let $c:X\times X\to X\times X$ be defined by $c(i,j)=(i\trid j,i)$.
The map $c$ is a solution of the braid equation. Since $X$ is
finite, we have 
$c^n=\mathrm{id}$ for some $n\in\ndN $ with $n\ge 1$.
Let $H$ be the group generated by $c$. The group
$H$ acts on $X\times X$ and the orbits of this action are
$\orbit(i,j)=\{c^{m}(i,j)\mid m\in\mathbb{Z}\}$, where $(i,j)\in X\times X$.

\begin{rem}
\label{rem:orbits}
The group $G$ acts by conjugation on $X$. Also, $G$ acts on $X\times X$ 
diagonally: $g\trid (x,y)=(g\trid x, g\trid y)$ 
for all $g\in G$ and $x,y\in X$.
For all $g\in G$ and all $x,y\in X$ we have $\orbit(g\trid
x,g\trid y)=g\trid\orbit(x,y)$.  In particular, 
for all $i,j\in X$ the
sets $\orbit(i,i\trid j)$ and $\orbit(j\trid i,j)$
have the same cardinality as $\orbit(i,j)$.  By definition we obtain that
$\#\orbit(i,j)=1$ if and only if $i=j$. Similarly, $\#\orbit(i,j)=2$
if and only if $i\ne j$ and $x_{i}$ and $x_{j}$ commute (i.e. $i\trid j=j$).
\end{rem}

The set $X\times X$ is the disjoint union of the orbits under
the action of $H$. For all $n\in\ndN $ let 
\[
l_{n}=\#\{\orbit(i,j):\orbit(i,j)\text{ has }n\text{ elements}\}.
\]
Since $l_1=d$, we obtain that
\begin{equation}
\label{eq:Lcounting_formula}
d+2l_{2}+3l_{3}+\cdots=d^{2}.
\end{equation}

For all $n\in\ndN $ let
\[ k_{n}=\#\{j\in X\mid\#\orbit(1,j)=n\}. \]
Remark \ref{rem:orbits} implies that 
\[
k_{n}=\#\{j\in X\mid\#\orbit(i,j)=n\}
\]
for all $i\in X$ and all
$n\in\ndN $.

\begin{lem}
\label{lem:general}
The rack $X$ has the following properties.
\begin{enumerate}
\item\label{lem:faithful}Let $i,j,k\in X$.
  If $\orbit(i,j)=\orbit(i,k)$ then $j=k$.
  If $\orbit(i,j)=\orbit(k,j)$ then $i=k$.
\item\label{lem:sizes}For all $i,j\in X$ we have $\#\orbit(i,j)\leq d$. 
In particular, $d\ge n$ for all $n\in \ndN $ with $k_n\geq 1$.
\item\label{lem:fixedpoints}For all $i\in X$ the permutation $\varphi_i$ has precisely $k_2+1$ fixed points. In particular,
if $k_2=d-1$ then $X$ is trivial (i.e. $i\trid j=j$ for all $i,j\in X$). 
\item\label{lem:counting_formula} $1+k_2+k_3+\cdots=d$.
\end{enumerate}
\end{lem}

\begin{proof}
(1) Since $c(i,j)=(i\trid j,i)=(x_{i}x_{j}x_{i}^{-1},x_{i})$
and $x_{i}x_{j}=(x_{i}x_{j}x_{i}^{-1})x_{i}$ in $G$, it follows that
$x_ix_j=x_lx_m$ for all $(l,m)\in\orbit(i,j)$.
In particular, 
if $(i,k)\in\orbit(i,j)$ then $x_{i}x_{k}=x_{i}x_{j}$ and hence $j=k$.
Similarly, if $\orbit (i,j)=\orbit (k,j)$ then $x_i x_j=x_k x_j$ and
hence $i=k$.

(2) follows from \eqref{lem:faithful}.

(3)
Let $i,j\in X$. Then $j$ is a fixed point of $\varphi_i$ if and only if $i=j$ 
or $\orbit(i,j)$ has size 2. This gives the claim.

(4)
By (1) the orbits $\orbit(1,i)$ with $1\leq i\leq d$ are disjoint.
Since $\#\orbit(1,1)=1$ and $2\leq\#\orbit(1,i)<\infty$ for all $i\geq 2$,
the claim follows.
\end{proof}

\begin{lem}
\label{lem:l_n}
For all $n\in\ndN $ we have $l_{n}=\frac{dk_{n}}{n}$. 
\end{lem}

\begin{proof}
Let $n\in\ndN $ and let $\Gamma_n=\{(i,j)\in X\times X\mid \#\orbit(i,j)=n\}$.
The set $\Gamma_n$ is invariant under the action of
$H$, and
every $H$-orbit of $\Gamma _n$
has size $n$ by definition of $\Gamma_n$. Hence $\#\Gamma_n=nl_n$.
On the other hand, $\Gamma_n$ is invariant under the diagonal action of $G$ by
Remark \ref{rem:orbits}. Since $G$ acts transitively on $X$, we conclude that
every $G$-orbit of $\Gamma_n$ has size $d$ and contains a unique element
$(1,j)$, where $j\in X$.
Hence $\#\Gamma_n=dk_n$. Therefore $nl_n=dk_n$.
\end{proof}

Let $x,y\in X$. For all $n\in\ndN $ define recursively $a_n,b_n\in X$ by
\begin{align*}
a_0=x,\qquad & b_0=y\\
a_1=x\trid y,\qquad & b_1=y\trid x\\
a_{n+1}=a_n\trid a_{n-1},\qquad & b_{n+1}=b_n\trid b_{n-1},\qquad n\ge 1.
\end{align*}

\begin{lem}
\label{lem:ab}
Let $x,y\in X$ and $n\in \ndN $ with $n\ge 1$. Then the following hold.
\begin{enumerate}
\item\label{item:c}$c^n(x,y)=(a_n,a_{n-1})$.
\item\label{item:axb}$a_n=x\trid b_{n-1}$.
\item\label{item:bya}$b_n=y\trid a_{n-1}$.
\end{enumerate}
\end{lem}

\begin{proof}
\eqref{item:c} follows by induction on $n$
from the definition of $c$, and \eqref{item:bya} is obtained from 
\eqref{item:axb} by exchanging $x$ and $y$.
We prove \eqref{item:axb} by induction on $n$. 
The case $n=1$ is trivial. The induction step follows from the equations
\[ x\trid b_n=x\trid (b_{n-1}\trid b_{n-2})=
(x\trid b_{n-1})\trid (x\trid b_{n-2})=a_n\trid a_{n-1}=a_{n+1} \]
which are obtained from Axiom~(R2) and the induction hypothesis.
\end{proof}

\begin{notation}
We write $x\blacktriangleright_n y=a_{n-1}$ and $y\blacktriangleright_n
x=b_{n-1}$ for all $n\in \ndN $ with $n\ge 1$. 
By Lemma \ref{lem:ab}(2),(3) we obtain that
\begin{gather}
x\blacktriangleright_n y=\underbrace{x\trid(y\trid(x\cdots))}_{n\text{ elements of }X}.
\end{gather}
For example we have
\begin{align*}
x\blacktriangleright_1 y=x&,\quad x\blacktriangleright_2 y=x\trid y,\\
x\blacktriangleright_3 y=x\trid (y\trid x)&,\quad x\blacktriangleright_4 y=x\trid (y\trid (x\trid y)).
\end{align*}
For any $x,y\in X$ the orbit $\orbit(x,y)$ will also be denoted by
$\obl x\obr $ if
$x=y$ and by
\[ \obl a_{m-2}~\cdots ~a_2~a_1~x~y\obr \]
if $x\not=y$, where $m\ge 2$ is the smallest integer with $c^m(x,y)=(x,y)$.
Lemma~\ref{lem:ab} implies that 
\begin{align} \label{eq:orbitsymm}
  \obl a_{m-2}~\cdots ~a_2~a_1~x~y\obr =
  \obl y~a_{m-2}~a_{m-3}~\cdots ~a_2~a_1~x\obr
\end{align}
and that the elements of $\orbit (x,y)$ with $x\not=y$
are just the pairs $(i,j)$ such that
$i$ and $j$ are consecutive entries in $\obl \cdots ~x~y\obr $ or its cyclic
permutation.
\end{notation}

\begin{lem}
  Assume that $X$ is involutive.
  Let $n\in \ndN $ with $n\ge 2$ and let $i_1,\dots,i_n\in X$ be pairwise
  distinct elements such that $\orbit (i_2,i_1)=\obl i_n~\cdots
  ~i_3~i_2~i_1\obr $. Then $\orbit (i_1,i_2)=\obl
  i_1~i_2~i_3~\cdots ~i_n\obr $.
  \label{lem:invrack}
\end{lem}

\begin{proof}
  By assumption we have $i_l\trid i_{l-1}=i_{l+1}$ and
  $c(i_l,i_{l-1})=(i_{l+1},i_l)$ for all $l\in
  \{1,2,\dots,n\}$, where the indices are considered as elements in $\ndZ
  /n\ndZ $.
  Since $X$ is involutive, it follows that $i_l\trid i_{l+1}=i_{l-1}$ for all
  $l\in \ndZ /n\ndZ $. This implies the claim.
\end{proof}

Recall that $G_X$ is the enveloping group of $X$. Since $G$ is generated
by $X$, there is a unique group epimorphism $G_X\to G$ induced by the
identity on $X$. Hence $X$ is an injective rack.
Let $\Phi:X\to G_X$ be defined by
$x\mapsto\overline{x}$, where $\overline{x}$ is the coset of $x$ in $G_X$.
Following the notation in \cite[page 7]{MR1778802}, for all
$v,w\in G_X$ and all $n\in\ndN $ we write 
\begin{gather*}
\mathrm{Prod}(v,w;0)=1,\quad
\mathrm{Prod}(v,w;1)=v,\quad
\mathrm{Prod}(v,w;2)=vw,\\
\mathrm{Prod}(v,w;n)=\underbrace{vwvw\cdots}_{n\text{ factors}}.
\end{gather*}

\begin{lem}
\label{lem:coxeter}
Let $x,y\in X$. In $G_X$ we have
\[
\overline{x\blacktriangleright_n y}=\mathrm{Prod}(\overline{x},\overline{y};n)\cdot\mathrm{Prod}(\overline{x},\overline{y},n-1)^{-1}
\]
for all integers $n\ge 1$.
\end{lem}

\begin{proof}
The case $n=1$ is trivial. 
Assume now that the claim holds for $n$. Then
\begin{align*}
\Phi(x\blacktriangleright_{n+1} y)&=\Phi(x\trid(y\blacktriangleright_n x))\\
&=\overline{x}\cdot\Phi(y\blacktriangleright_n x)\cdot\overline{x}^{-1}\\
&=\overline{x}\cdot\mathrm{Prod}(y,x;n-1)\cdot\mathrm{Prod}(y,x;n-2)^{-1}\cdot\overline{x}^{-1}\\
&=\mathrm{Prod}(x,y;n)\cdot\mathrm{Prod}(x,y,n-1)^{-1}.
\end{align*}
This completes the proof.
\end{proof}

\begin{prop}
\label{prop:orbits}
Let $x,y\in X$ and $n\in\ndN $ with $n\ge 1$.
The following are equivalent.
\begin{enumerate}
\item\label{item:orbit_size}The $H$-orbit $\orbit(x,y)$ has size $n$.
\item\label{item:xBy=y}$x\blacktriangleright_{n} y=y$ and
  $x\blacktriangleright_k y\ne y$ for all $k\in\{1,2,\dots ,n-1\}$.
\item\label{item:yBx=x}$y\blacktriangleright_{n} x=x$ and
  $y\blacktriangleright_k x\ne x$ for all $k\in\{1,2,\dots ,n-1\}$.
\item\label{item:prod}$\mathrm{Prod}(\overline{x},\overline{y};n)=
\mathrm{Prod}(\overline{y},\overline{x};n)$ and 
$\mathrm{Prod}(\overline{x},\overline{y};k)\ne\mathrm{Prod}(\overline{y},\overline{x};k)$
for all $k\in\{1,2,\dots ,n-1\}$.
\end{enumerate}
\end{prop}

\begin{proof}
Recall that \eqref{item:orbit_size} holds if and only if $c^n(x,y)=(x,y)$ and
$c^k(x,y)\ne(x,y)$ for all $k\in\{1,2,\dots ,n-1\}$.
Thus \eqref{item:orbit_size} is equivalent to \eqref{item:xBy=y} by Lemma \ref{lem:ab}\eqref{item:c} 
and Lemma \ref{lem:general}\eqref{lem:faithful}.
The equivalence between \eqref{item:xBy=y} and \eqref{item:prod} follows from
Lemma \ref{lem:coxeter}.
Exchanging $x$ and $y$ in the last argument one concludes that
\eqref{item:prod} and \eqref{item:yBx=x} are
equivalent, which finishes the proof of the proposition.
\end{proof}

\begin{cor}
\label{cor:orbit3}
Let $x,y,z\in X$. Assume that the orbit $\orbit(x,y)$ has size 3. If $x\trid y=z$ 
then $y\trid z=x$ and $z\trid x=y$. 
\end{cor}

\begin{cor}
\label{cor:orbit4}
Let $x,y\in X$. Assume that the orbit $\orbit(x,y)$ has size 4. If $X$ is involutive
then $x\trid(y\trid x)=y\trid x$ and
$y\trid(x\trid y)=x\trid y$.
\end{cor}

Now we start to study in more detail the racks satisfying the inequality 
\begin{gather}
\label{eq:N1}
d+l_{2}+l_{3}+\cdots\geq\frac{d(d-1)}{2}.
\end{gather}
This condition is motivated by the structure of the degree two part of the Nichols algebra 
of $(X,q)$, where $q$ is a two-cocycle, see Section \ref{YD}.
Our main achievement in this section
will be the classification of these racks, see 
Theorem \ref{thm:racks} below.

Lemma \ref{lem:l_n} and Lemma~\ref{lem:general}\eqref{lem:counting_formula}
imply that \eqref{eq:N1} is equivalent to 
\begin{equation}
\sum_{n\ge3}\frac{n-2}{2n}k_{n}\leq1.\label{eq:QRcondition}
\end{equation}

\begin{notation}
We write $\mathcal{S}=\sum_{n\geq3}\frac{n-2}{2n}k_n$.
\end{notation}

\begin{rem}
\label{rem:kn_bounds}If \eqref{eq:QRcondition} is satisfied,
then $k_{n}\leq\frac{2n}{n-2}$ for all $n\geq3$. 
Since $k_n\in\mathbb{Z}$ for all $n$, 
this means that 
$k_{3}\leq6$, $k_{4}\leq4$, $k_{5}\leq3$, $k_{6}\leq3$
and $k_{n}\leq2$ for all $n>6$. 
\end{rem}

\begin{lem}
\label{lem:about_k2}
Let $m=d-k_2-1$. If Relation \eqref{eq:QRcondition} holds then $2\leq m\leq 6$.
\end{lem}

\begin{proof}
If $k_{2}=d-1$ then $X$ is trivial. Since $X$ is indecomposable, this is a contradiction 
to $d>1$. The case $k_{2}=d-2$ is impossible by Lemma \ref{lem:general}\eqref{lem:fixedpoints}.
Therefore $m\geq2$. 
Relation \eqref{eq:QRcondition} and
Lemma~\ref{lem:general}\eqref{lem:counting_formula}
imply that
\[
\sum_{n\geq3}\frac{k_n}{n}\geq\sum_{n\geq3}\frac{k_n}{2}-1
=\frac{d-3-k_2}{2}=\frac{m-2}{2}.
\]
Thus 
\[
\frac{m-2}{4}\leq\sum_{n\geq3}\frac{1}{2n}k_{n}\leq
\sum_{n\geq 3}\frac{n-2}{2n}k_{n}\leq 1.
\]
Therefore $m\leq6$. 
\end{proof}


\begin{thm}
\label{thm:racks}
Let $G$ be a group and let $X$ be a conjugacy class of $G$. 
Assume that $X$ is finite and generates $G$. The following are equivalent.
\end{thm}
\begin{enumerate}
\item $\sum_{n\ge3}\frac{n-2}{2n}k_{n}\leq1.$
\item The rack $X$ is one of the racks listed in Table \ref{tab:racks}.
\end{enumerate}

\begin{proof}
First we prove that $(1)$ implies $(2)$. 
Since $X$ generates $G$, Lemma \ref{lem:about_k2} yields that 
$k_{3}+k_{4}+\cdots=m$, where
$2\leq m\leq6$. We split the proof in several subsections. The rack $\mathbb{D}_3$
appears when $m=2$, see Subsection \ref{sub:case_m=00003D3}.
The rack $\mathcal{T}$ appears in 
Subsection \ref{sub:case_m=00003D4}. The affine racks of $\mathbb{F}_{5}$
(resp. $\mathbb{F}_{7}$) appear in Subsection \ref{sub:case_m=00003D5}
(resp. \ref{sub:case_m=00003D7}). The rack $\mathcal{B}$ 
appears in Subsection \ref{sub:case_m=00003D5}.
The rack $\mathcal{A}$ 
(resp. $\mathcal{C}$) 
appears in Subsection \ref{sub:case_m=00003D5}
(resp. \ref{sub:case_m=00003D7}).

Now we prove that $(2)$ implies $(1)$. These racks are explained in Examples
\ref{exa:racks} and \ref{exa:affine}. The computations appear 
inside Subsections \ref{sub:case_m=00003D3}, \ref{sub:case_m=00003D4}, 
\ref{sub:case_m=00003D5} and \ref{sub:case_m=00003D7}. 
Also, these computations can be easily done with \textsf{GAP} 
\cite{GAP, RIG}. It turns out that $k_n=0$ in all examples, where $n>5$.
In Table~\ref{tab:racks} we record all numbers
$k_n$ with $2\le n\le 4$.
\end{proof}

\begin{table}[ht]
\caption{Racks satisfying Condition \eqref{eq:QRcondition}}
\begin{center}
\label{tab:racks}
\begin{tabular}{|c|c|c|c|c|c|c|}
\hline 
Rack & $d$ & $k_{2}$ & $k_{3}$ & $k_{4}$ & $\mathcal{S}$\tabularnewline
\hline 
$\mathbb{D}_{3}$ & 3 & 0 & 2 & 0 & $\frac{1}{3}$\tabularnewline
\hline 
$\mathcal{T}$ & 4 & 0 & 3 & 0 & $\frac{1}{2}$\tabularnewline
\hline 
$\mathrm{Aff}(5,2)$ & 5 & 0 & 0 & 4 & $1$\tabularnewline
\hline 
$\mathrm{Aff}(5,3)$ & 5 & 0 & 0 & 4 & $1$\tabularnewline
\hline 
$\mathcal{A}$ & 6 & 1 & 4 & 0 & $\frac{2}{3}$\tabularnewline
\hline 
$\mathcal{B}$ & 6 & 1 & 4 & 0 & $\frac{2}{3}$\tabularnewline
\hline 
$\mathrm{Aff}(7,3)$ & 7 & 0 & 6 & 0 & $1$\tabularnewline
\hline 
$\mathrm{Aff}(7,5)$ & 7 & 0 & 6 & 0 & $1$\tabularnewline
\hline 
$\mathcal{C}$ & 10 & 3 & 6 & 0 & $1$\tabularnewline
\hline
\end{tabular}
\par\end{center}
\end{table}

\section{The proof of Theorem \ref{thm:racks}}

This section is devoted to the proof of Theorem \ref{thm:racks}.

\subsection{The case $k_{3}+k_{4}+\cdots=2$.\label{sub:case_m=00003D3}}

By Lemma \ref{lem:general}\eqref{lem:fixedpoints},\eqref{lem:counting_formula}
every permutation $\varphi_{i}$, where $i\in X$, 
is a transposition. 
Without loss of generality 
we may assume that $\varphi_1=(2\,3)$. 
Hence $\#\orbit(1,i)>2$ if and only if $i\in\{2,3\}$. 
Moreover, 
\[
\#\orbit(1,3)=\#(1\trid\orbit(1,2))=\#\orbit(1,2)
\]
by Remark \ref{rem:orbits}.
Let $n=\#\orbit(1,2)$. 
Then
\begin{gather}
\label{case:m=2}
k_n=2 \text{ and }k_m=0 \text{ for all $m\ge 3$, $m\ne n$}.
\end{gather}
Lemma \ref{lem:general}\eqref{lem:sizes} implies that $d\geq n$.
We split the proof in three steps.

\textit{Step 1.} $n=3$.
The definition of $c$ implies that
\[ \orbit(1,2)=\obl 3~1~2\obr . \]
Hence $\varphi_{1}=(2\,3)$, $\varphi_{2}=(1\,3)$ and $\varphi_{3}=(1\,2)$.
Thus $X=\{1,2,3\}$ by Lemma~\ref{lem:cs_indecomposable}
with $Y_1=Y_2=\{1,2,3\}$.
Therefore $X$ is the dihedral rack of 3 elements.

\textit{Step 2.} $n=4$.
Corollary \ref{cor:orbit4} with $(x,y)=(1,3)$
implies that $1\trid(3\trid1)=3\trid1$. 
Since $1\trid2\ne2$ and $1\trid3\ne3$ we conclude that
$3\trid1\notin\{1,2,3\}$. Let $4=3\trid1$. Then
\[
\orbit (1,2)= \obl 4~3~1~2 \obr .
\]
Therefore $\varphi_{4}=(2\,3)$ and $\varphi_{2}=\varphi_{3}=(1\,4)$. 
Since $X$ is indecomposable, this is a contradiction to
Lemma \ref{lem:cs_indecomposable} with $Y_1=\{2,3\}$, $Y_2=\{1,4\}$.

\textit{Step 3.} $n>4$. Let $y_4,y_5,\dots,y_n\in X$ such that
\[ \orbit (1,2)=\obl y_n~\dots y_5~y_4~3~1~2 \obr .\]
On the one hand, by acting with $1$ we obtain that
\[ \orbit (1,3)=\obl y_n~\dots y_5~y_4~2~1~3 \obr .\]
On the other hand,
$\orbit (1,3)=\obl y_4~y_5~\dots y_n~2~1~3 \obr $ by
Lemma~\ref{lem:invrack} and Equation~\eqref{eq:orbitsymm}.
This is a contradiction to $y_4\not=y_n$.

\subsection{The case $k_{3}+k_{4}+\cdots=3$\label{sub:case_m=00003D4}}

By Lemma \ref{lem:general}\eqref{lem:fixedpoints},\eqref{lem:counting_formula}
we may assume 
without loss of generality that $\varphi_{1}=(2\,3\,4)$. 
Let $n=\#\orbit(1,2)$. 
By Remark \ref{rem:orbits} we conclude that
\[
n=\#\orbit(1,3)=\#\orbit(1,4).
\]
Therefore
$k_n=3$ and $k_m=0$ for all $m\ge 3$ with $m\not=n$. 
Remark~\ref{rem:kn_bounds} implies that 
$n\in\{3,4,5,6\}$.
We split the proof in two steps.

\textit{Step 1.} $n=3$. Then $\orbit (1,2)=\obl 3~1~2\obr $, $\orbit
(1,3)=\obl 4~1~3\obr $, and $\orbit (1,4)=\obl 2~1~4\obr $.
Since all permutations $\varphi _i$ with $i\in X$ have the same cycle
structure, we conclude that
\[
\varphi_{2}=(3\,1\,4),\quad
\varphi_{3}=(4\,1\,2),\quad
\varphi_{4}=(2\,1\,3).
\]
Since $X$ is indecomposable,
Lemma \ref{lem:cs_indecomposable} with $Y_1=Y_2=\{1,2,3,4\}$
implies that $d=4$. 
Then $X$ is the rack $\mathcal{T}$,
the rack associated to
the vertices of the tetrahedron. 
The isomorphism $X\to\mathcal{T}$ is given by
$i\mapsto\pi_i$.

\textit{Step 2.} $n\ge 4$. Since $1\trid 2\not=2$, it follows that
$2\trid 1\not=1$, and hence there exist $i,j\in X$ with $\varphi
_2=(1\,i\,j)$. Since $n\ge 4$, we conclude that $\orbit (1,4)=
\obl \cdots ~i~2~1~4\obr $ and $i\notin \{1,2,4\}$, $j\notin \{1,2,i\}$.
Up to renaming the variables we may assume that $i\in \{3,5\}$.

\textit{Step 2.1.}
Assume that $\varphi_2=(1\,3\,j)$. 
By conjugation with $\varphi_1$ and by Equation~\eqref{eq:conjugation}
we obtain that
$\varphi_3=(1\,4\,(1\trid j))$.
In the same way, since $2\trid 1=3$, we conclude that
$\varphi_3=\varphi_2\varphi_1\varphi_2^{-1}
=(2\,j\,(2\trid 4))$. Comparing the two formulas for $\varphi _3$ implies that
$\varphi _3=(1\,4\,2)=(2\,1\,4)$, a contradiction to $j\not=1$.

\textit{Step 2.2.}
Assume that $\varphi_2=(1\,5\,j)$, where $j\notin \{1,2,5\}$. 
By conjugation with $\varphi_1$ 
we obtain that $\varphi_3=(1\,5\,(1\trid j))$.
Since $\varphi _3^2\varphi _2(1)=1$,
Lemma~\ref{lem:conjugation}
gives that $\varphi _3^2\varphi _2\varphi _1=\varphi _1\varphi _3^2\varphi
_2$. Evaluating this equation at $5$ yields that
\begin{align} \label{eq:133j=33j}
  \varphi _3^2(j)=\varphi _1\varphi _3^2(j).
\end{align}
Thus $\varphi _3^2(j)\notin \{2,3,4\}$.
Therefore $1\trid j=j$, since otherwise $\varphi _3(j)=j$ which
would be a contradiction to Equation~\eqref{eq:133j=33j}.
Conjugation of $\varphi _2$ by powers of $\varphi _1$ yields that
$\varphi_2=\varphi_3=\varphi_4=(1\,5\,j)$ and $j\notin \{1,2,3,4,5\}$.
Since $X$ is indecomposable, this is a contradiction
to Lemma~\ref{lem:cs_indecomposable} with
$Y_1=\{2,3,4\}$, $Y_2=\{1,5,j\}$.

\subsection{The case $k_{3}+k_{4}+\cdots=4$\label{sub:case_m=00003D5}}

By Lemma \ref{lem:general}\eqref{lem:fixedpoints},\eqref{lem:counting_formula}
we know that $\varphi_{1}$ has $d-4$
fixed points. Therefore, we have to consider two cases: $\varphi_{1}=(2\,3)(4\,5)$
or $\varphi_{1}=(2\,3\,4\,5)$.

\textit{Step 1.} $\varphi_{1}=(2\,3)(4\,5)$.

By Remark \ref{rem:orbits} we may assume that
$\#\orbit(1,2)=\#\orbit(1,3)=p$ and
$\#\orbit(1,4)=\#\orbit(1,5)=q$, where $3\leq p\leq q$. If $p=q$ then $k_{p}=4$ and then
$p\in\{3,4\}$ by Remark \ref{rem:kn_bounds}. 
If $p<q$ then $k_p=k_q=2$.
Hence $p\leq4$ by Relation~\eqref{eq:QRcondition}.
A closer look at Relation \eqref{eq:QRcondition} gives that one of the
following holds.
\begin{itemize}
\item $k_{3}=4$ and $k_{n}=0$ for $n>3$.
\item $k_{4}=4$ and $k_{n}=0$ for $n\ge 3$, $n\not=4$.
\item $k_{3}=k_{q}=2$, where $q\in \{4,5,6\}$, and $k_{n}=0$ for $n\ge 4$,
  $n\not=q$.
\end{itemize}
We divide the classification of these cases into two steps. In the first
step we assume that $k_3\ge 2$ and in the second step we consider the
remaining case when $k_4=4$.

\textit{Step 1.1.} $k_3\geq 2$.
Then $\orbit(1,2)=\obl 3~1~2\obr $
and $\orbit(1,3)=\obl 2~1~3\obr $, and hence up to renaming
we may assume that $\varphi _2\in \{(1\,3)(4\,5),(1\,3)(4\,6),(1\,3)(6\,7)\}$.
However, if
$2\trid6=7$ then
\begin{align*}
  7=1\trid7=(2\trid 3)\trid 7\overseteqnum{eq:conjugation}{=}
  2\trid (3\trid 6)=2\trid ( (1\trid 2)\trid 6)
  \overseteqnum{eq:conjugation}{=}2\trid (1\trid (2\trid 6))=6,
\end{align*}
a contradiction.

Next we show that $\varphi_2\not=(1\,3)(4\,5)$. Indeed, assume that $2\trid
4=5$ and let $i=4\trid 1$. Then $4\trid i=1$ since $X$ is involutive.
Moreover,
$\varphi _3=\varphi _{2\trid 1}=(1\,2)(4\,5)$
by Equation~\eqref{eq:conjugation}.
First, $i\not=1$ by Remark~\ref{rem:crossed_set}
since $1\trid 4\not=4$.
Second, Lemma \ref{lem:involutions_abc} with $(a,b,c)=(1,4,1\trid i)$
implies that $i\notin\{2,3\}$. Third, $i\not=4$ since $\varphi _4$ is
injective and $4\trid 4=4$.
Finally, Lemma \ref{lem:racks_abcd} with	$(a,b,c,d)=(1,2,4,i)$ implies that
$i\notin\{5,6\}$, a contradiction.

By the above we conclude that $\varphi_2=(1\,3)(4\,6)$, and hence
Equation~\eqref{eq:conjugation} gives that
$\varphi_3=\varphi _2\varphi _1\varphi _2^{-1}=(1\,2)(5\,6)$. 

\begin{lem}
\label{lem:auxiliar2_m=4}
The orbit $\orbit(1,4)$ has size 3.
\end{lem}

\begin{proof}
Assume that $\#\orbit(1,4)>3$. 
Then $4\trid1\ne5$. Indeed, if $4\trid1=5$ then $4\trid 5=1$ and hence
$\orbit(1,4)=\obl 5~1~4\obr $, a contradiction.

Remark \ref{rem:crossed_set} and (R1) yield that $4\trid 1\notin \{1,4\}$.
Moreover, $4\trid1\not=3$. Indeed, otherwise $4\trid 3=1$ since $X$ is
involutive. However, $3\trid4=4$ and Remark~\ref{rem:crossed_set}
imply that $4\trid 3=3$, a contradiction.

By the above we conclude that $4\trid1\in\{2,6,7\}$.
If $4\trid1=2$ then 
\[ 6\trid3=\varphi_2(4\trid1)=2,\quad
5\trid1=\varphi_1(4\trid1)=3. \]
Then Lemma \ref{lem:rack_indecomposable} with $Y=\{1,2,3\}$ yields
a contradiction since
$X$ is indecomposable.
Lemma \ref{lem:racks_commute} with $(a,x,y,z)=(3,4,1,7)$ implies
that $4\trid1\ne 7$. 
Finally, if $4\trid1=6$ then
$4\trid2=\varphi_3(4\trid1)=5$ and hence
\[ \varphi _4=(1\,6)(2\,5),\quad
   \varphi _5=(1\,6)(3\,4),\quad
   \varphi _6=(3\,4)(2\,5)
\]
by Equation~\eqref{eq:conjugation} using conjugation with $\varphi _1$
and $\varphi _2$, respectively. Then
\[ \varphi _4\varphi _5\varphi _4^{-1}= (1\,6)(3\,4)\not=\varphi _2,\]
a contradiction to
Equation~\eqref{eq:conjugation}. Thus $\# \orbit (1,4)=3$.
\end{proof}

Lemma \ref{lem:auxiliar2_m=4} gives that $k_3=4$
and $k_n=0$ for all $n>3$. Thus
$\#\orbit(1,i)=3$ for all $i\in \{2,3,4,5\}$ and hence
\[ \orbit (1,2)=\obl 3~1~2\obr,\quad
   \orbit (1,3)=\obl 2~1~3\obr,\quad
   \orbit (1,4)=\obl 5~1~4\obr,\quad
   \orbit (1,5)=\obl 4~1~5\obr . \]
Then $4\trid2=(3\trid4)\trid(3\trid1)=3\trid(4\trid1)=3\trid5=6$.
We have
$\varphi_{4}=(1\,5)(2\,6)$, $\varphi_{5}=(1\,4)(3\,6)$ and
$\varphi_{6}=(2\,4)(3\,5)$. Since $X$ is indecomposable, Lemma
\ref{lem:cs_indecomposable} with $Y_1=Y_2=\{1,2,\dots,6\}$ implies that $d=6$.
This rack is isomorphic to $\mathcal{A}$. The isomorphism $X\to\mathcal{A}$ 
is given by $i\mapsto\pi_{p(i)}$, where $p=(1\,2)(4\,5)$.

\textit{Step 1.2.} $k_4=4$ and $k_n=0$ for all $n\ge 3$ with $n\not=4$.
Then the orbits $\orbit(1,2)$ and $\orbit(1,3)$ have size $4$.
Since $X$ is involutive, we conclude from
Corollary \ref{cor:orbit4} that $1$ and $2\trid 1$ commute, and hence we may
assume that $2\trid1=6$. Conjugation by $\varphi _1$ yields that $3\trid 1=6$,
and hence
\[ \orbit (1,2)=\obl 6~3~1~2\obr,\quad
   \orbit (1,3)=\obl 6~2~1~3\obr .
\]
Then $2\trid \orbit (1,2)=\obl 1~2\trid 3~6~2\obr =\obl 6~2~1~2\trid 3\obr $,
that is, $2\trid 3=3$.
Up to renaming we obtain that
$\varphi _2\in \{(1\,6)(4\,5),(1\,6)(4\,7),(1\,6)(7\,8)\}$.
The same arguments applied to $\orbit(1,4)$ and $\orbit (1,5)$, which also have
size $4$,
give that $4\trid 1=5\trid 1\in X\setminus \{1,2,3,4,5\}$ and
$4\trid 5=5$.
If $4\trid 1=6$ then Lemma \ref{lem:rack_indecomposable}
with $Y=\{1,6\}$ gives a contradiction
since $X$ is indecomposable.
Therefore
\begin{align} \label{eq:4trid1}
  4\trid 1=5\trid 1\in X\setminus \{1,2,3,4,5,6\}.
\end{align}

\begin{caseenv}
\item Assume that $\varphi_2=(1\,6)(4\,5)$.
  Then $\varphi_3=\varphi _{1\trid 2}=\varphi _2$
  and $\varphi_6=\varphi_{2\trid 1}=\varphi_1$ by
  Equation~\eqref{eq:conjugation}.
  By the above we may assume that $4\trid1=7$. Since
  neither $2$ nor $3$ commutes with $4$,
  we conclude that
  $\varphi_4=(1\,7)(2\,3)$. 
  Then $\varphi_5=\varphi_{1\trid 4}=\varphi_4$
  and $\varphi_7=\varphi_{4\trid 1}=\varphi_1$.
  Then $X\trid \{4,5\}\subseteq \{4,5\}$ which is a contradiction
  to Lemma \ref{lem:rack_indecomposable}.

\item Assume that $\varphi_2=(1\,6)(7\,8)$. Then
  $\varphi_3=\varphi_{1\trid 2}=\varphi_2$
  and $\varphi _6=\varphi _{2\trid 1}=\varphi _1$.
  Since $2\trid7\ne 7$ and $3\trid 7\ne 7$, we conclude that
  $\varphi_7\in \{(2\,3)(i\,j),(2\,i)(3\,j)\}$ for some $i,j\in X\setminus
  \{2,3,7\}$ with $i\not=j$.
  If $7\trid2=3$ then $\varphi _8=\varphi _{2\trid 7}=(2\,3)(2\trid i\,2\trid
  j)$, and hence $X\trid \{2,3\}\subseteq \{2,3\}$ in contradiction to Lemma
	\ref{lem:rack_indecomposable}.
	Therefore $\varphi_7=(2\,i)(3\,j)$. 
	Applying $\varphi_1$ we conclude that $i,j\in\{4,5\}$. 
	Let $k\in\{2,3\}$ such that $7\trid k=4$. Then 
	$4\trid 6=(7\trid k)\trid(7\trid6)=
	7\trid(k\trid6)=7\trid1=1$, which contradicts to \eqref{eq:4trid1}.
\item Assume that $\varphi_2=(1\,6)(4\,7)$.
  It follows that $\varphi_3=\varphi _{1\trid 2}=(1\,6)(5\,7)$ and
  $\varphi_6=\varphi _{2\trid 1}=(2\,3)(5\,7)$. Moreover, $\varphi _6=\varphi
  _{3\trid 1}=(2\,3)(4\,7)$, a contradiction.
\end{caseenv}

We conclude that there are no racks satisfying the assumption in Step~1.2.
The only solution of Step~1 is the rack $\mathcal{A}$.

\vspace{\baselineskip}

\textit{Step 2.} $\varphi_ 1=(2\,3\,4\,5)$.
Let $m=\#\orbit(1,2)$. Remark \ref{rem:orbits} implies that
$k_m=4$ and that $k_n=0$ for all $n\ge 3$ with $n\not=m$.
By Remark \ref{rem:kn_bounds} it follows that $m\in \{3,4\}$.

\textit{Step 2.1.} $k_3=4$ and $k_n=0$ for all $n>3$.
Then
\[ \orbit (1,2)=\obl 3~1~2\obr,\quad
\orbit (1,3)=\obl 4~1~3\obr,\quad
\orbit (1,4)=\obl 5~1~4\obr,\quad
\orbit (1,5)=\obl 2~1~5\obr .
\]
Therefore
$\varphi_2=(3\,1\,5\,i)$ for some $i\in X$, and without loss of generality we
may assume that $i\in \{4,6\}$.

\begin{caseenv}
\item If $2\trid5=4$ then $\varphi _3=\varphi _{1\trid 2}=
  (4\,1\,2\,5)$ and hence $\varphi _2=\varphi _{3\trid 1}=(5\,3\,1\,4)$,
  a contradiction.
\item If $2\trid5=6$ then
  $\varphi_2=(1\,5\,6\,3)$,
  and by applying $\varphi_1$ we obtain that
  $\varphi_3=(1\,2\,6\,4)$,
  $\varphi_{4}=(1\,3\,6\,5)$, $\varphi_{5}=(1\,4\,6\,2)$.
  Moreover, $\varphi _6=\varphi _{2\trid 5}=(2\,5\,4\,3)$.
  Since $X$ is indecomposasble, Lemma \ref{lem:cs_indecomposable} with 
  $Y_1=Y_2=\{1,2,\dots ,6\}$ implies that $d=6$. 
  This rack is 
  $\mathcal{B}$. 
\end{caseenv}

\textit{Step 2.2.} $k_{4}=4$ and $k_{n}=0$ for all $n\ge 3$ with $n\not=4$.
Since $1\trid 2=3$, it follows that
$\orbit (1,2)=\obl i~3~1~2\obr $ for some $i\in X\setminus \{1,2,3\}$.
We may assume that $i\in\{4,5,6\}$.

\begin{caseenv}
\item $3\trid1=4$. Then
  $\orbit (1,2)=\obl 1~2~4~3\obr $, and
  by conjugation with $\varphi _1$ we conclude that
\[
  \orbit (1,3)=\obl 1~3~5~4\obr , \quad
  \orbit (1,4)=\obl 1~4~2~5\obr , \quad
  \orbit (1,5)=\obl 1~5~3~2\obr .
\]
Therefore $\varphi_2=(5\,4\,1\,3)$. Conjugation with $\varphi _1$
yields that
$\varphi_3=(2\,5\,1\,4)$, $\varphi_{4}=(3\,2\,1\,5)$ and
$\varphi_{5}=(4\,3\,1\,2)$.
Since $X$ is indecomposable, Lemma \ref{lem:cs_indecomposable}
with $Y_1=Y_2=\{1,2,3,4,5\}$ implies that $d=5$. 
This rack is the affine rack associated with $(\mathbb{F}_{5},3)$.

\item $3\trid1=5$. As in the previous case we obtain that
  \begin{align*}
    \orbit (1,2)=&\obl 1~2~5~3\obr , &
    \orbit (1,3)=&\obl 1~3~2~4\obr ,\\
    \orbit (1,4)=&\obl 1~4~3~5\obr , &
    \orbit (1,5)=&\obl 1~5~4~2\obr 
  \end{align*}
  and that
  \begin{align*}
    \varphi_2=(5\,1\,4\,3),\quad
    \varphi_3=(2\,1\,5\,4),\quad
    \varphi_4=(3\,1\,2\,5),\quad
    \varphi_5=(4\,1\,3\,2).
  \end{align*}
  Since $X$ is indecomposable,
  Lemma \ref{lem:cs_indecomposable} with $Y_1=Y_2=\{1,2,\dots ,5\}$
  implies that
  $d=5$.
  This rack is the affine rack associated to $(\mathbb{F}_{5},2)$.

\item $3\trid1=6$. Then $\orbit (1,2)=\obl 1~2~6~3\obr $.
  Conjugation with $\varphi _1$ gives that
  $\orbit (1,3)=\obl 1~3~6~4\obr $. Hence $3\trid 1=6$ and $3\trid 6=1$,
  a contradiction since $\varphi _3$ is a $4$-cycle.
\end{caseenv}

\subsection{The case $k_{3}+k_{4}+\cdots=5$\label{sub:case_m=00003D6}}

By Lemma \ref{lem:general}\eqref{lem:fixedpoints},\eqref{lem:counting_formula}
we may assume that $\varphi_{1}\in \{(2\,3\,4\,5\,6),
(2\,3)(4\,5\,6)\}$.

\textit{Step 1.} $\varphi_{1}=(2\,3\,4\,5\,6)$.
Remark \ref{rem:orbits} implies that $k_m=5$ for some $m\ge 3$ and that
$k_n=0$ for all $n\ge 3$ with $n\not=m$.
By Remark~\ref{rem:kn_bounds} we obtain that $m=3$.
Thus $\#\orbit(i,j)=3$ for all $i,j\in X$ with $i\trid j\not=j$.
As in Subsection~\ref{sub:case_m=00003D5}, Step 2.1.\ we conclude that
$\varphi_2=(3\,1\,6\,\cdots)$, $\varphi_3=(4\,1\,2\,\cdots)$,
$\varphi_4=(5\,1\,3\,\cdots)$, $\varphi_5=(6\,1\,4\,\cdots)$
and $\varphi_6=(2\,1\,5\,\cdots)$.
Therefore $2\trid6\in\{4,5,7\}$.  Since
\begin{gather}
\label{eq:case1_m=5}
6\trid(2\trid6)=(6\trid2)\trid6=1\trid6=2
\end{gather}
and $\varphi _6$ is a $5$-cycle, it follows that
$2\trid6\ne5$.  If $2\trid6=7$, applying
$\varphi_1$ we have that $3\trid2=4\trid3=5\trid4=6\trid5=7$, which contradicts
\eqref{eq:case1_m=5} since $\varphi _6$ is a $5$-cycle.
Therefore $2\trid6=4$ and by applying $\varphi_1$ we obtain that
$3\trid2=5$, $4\trid3=6$, $5\trid4=2$ and $6\trid5=3$. 
Moreover, $2\trid 5\ne 5$ by Remark \ref{rem:crossed_set} and since
$5\trid 2\ne2$.
Hence
$\varphi_{2}=(3\,1\,6\,4\,5)$, $\varphi_{3}=(1\,2\,5\,6\,4)$,
$\varphi_{4}=(1\,3\,6\,2\,5)$, $\varphi_{5}=(1\,4\,2\,3\,6)$ and
$\varphi_{6}=(1\,5\,3\,4\,2)$. However there is no such rack,
since for example $3\trid(6\trid4)\ne(3\trid6)\trid(3\trid4)$.

\textit{Step 2.} $\varphi_{1}=(2\,3)(4\,5\,6)$.
Let $p=\#\orbit(1,2)$ and let $q=\#\orbit(1,4)$. Since $1\trid 2\not=2$ and
$1\trid 4\not=4$, it follows that $p,q\geq3$. 
Remark \ref{rem:orbits} implies that $k_q\geq3$. 
If $p=q$ then $k_p=5$, and therefore $p=3$ by Remark \ref{rem:kn_bounds}.
If $p\ne q$ then
$k_p=2$ and $k_q=3$, since $k_3+k_4+\cdots=5$. In this case
$p=4$, $q=3$ by \eqref{eq:QRcondition}. Indeed, if $p<q$ then
$\frac{p-2}{2p}k_p+\frac{q-2}{2q}k_q\ge 
\frac{3-2}{6}2+\frac{4-2}{8}3=\frac{13}{12}>1$, and if $q<p$ then
$\frac{p-2}{2p}k_p+\frac{q-2}{2q}k_q\ge 
\frac{4-2}{8}2+\frac{3-2}{6}3=1$, and equality holds if and only if $q=3$,
$p=4$.
This means that 
\begin{gather}
  \orbit (1,4)=\obl 5~1~4 \obr,\quad
  \orbit (1,5)=\obl 6~1~5 \obr,\quad
  \orbit (1,6)=\obl 4~1~6 \obr.
\end{gather}
Hence
$\varphi_4=(5\,1\,6)(\cdot\cdot)$, $\varphi_5=(6\,1\,4)(\cdot\cdot)$ and
$\varphi_6=(4\,1\,5)(\cdot\cdot)$.
Since $\varphi_1^3=(2\,3)$, it follows that $\varphi _1^3(6)=6$ and hence
$\varphi _1^3\varphi _6 \varphi _1^{-3}=\varphi _6$. Therefore
$\varphi_6=(1\,5\,4)(2\,3)$ or $\varphi_6=(1\,5\,4)(7\,8)$.

Assume first that $\varphi_6=(1\,5\,4)(7\,8)$.
Then $6\trid2=2$, and hence $2\trid6=6$
by Remark \ref{rem:crossed_set}. 
By applying $\varphi_1^2=(4\,6\,5)$ twice
we obtain that
$2\trid5=5$ and $2\trid4=4$.
Then
\[
4=2\trid4=2\trid(5\trid1)=(2\trid5)\trid(2\trid1)
=5\trid(2\trid1),
\]
and hence
$2\trid1=1$, which contradicts $1\trid2=3$ and Remark~\ref{rem:crossed_set}.

Assume now that $\varphi_6=(1\,5\,4)(2\,3)$. Then $2\trid 6\not=6$,
and by conjugation with $\varphi _1^2$ we obtain that $2\trid 5\not=5$,
$2\trid 4\not=4$.
Thus there exists $i\in X$ such that $\varphi _2$ permutes $\{1,i,4,5,6\}$.
Since $\varphi _1^2(2)=2$, Lemma~\ref{lem:conjugation} gives that
$\varphi _1^2\varphi_2\varphi_1^{-2}=\varphi _2$. Therefore
$\varphi _2\in \{(1\,i)(4\,5\,6),(1\,i)(4\,6\,5)\}$, and we obtain a
contradiction to $\varphi _6^2\varphi _2\varphi
_6^{-2}=\varphi _2$, where the latter holds
by Lemma~\ref{lem:conjugation} and since
$\varphi _6^2(2)=2$.

\subsection{The case $k_{3}+k_{4}+\cdots=6$\label{sub:case_m=00003D7}}

By Lemma \ref{lem:general}\eqref{lem:fixedpoints},\ref{lem:counting_formula}
we obtain that $\varphi_{1}\in \{(2\,3\,4\,5\,6\,7),
(2\,3)(4\,5\,6\,7),(2\,3\,4)(5\,6\,7),(2\,3)(4\,5)(6\,7)\}$.

\begin{lem}
\label{lem:solutions_m=6}
We have $k_{3}=6$ and $k_{n}=0$ for all $n>3$. 
\end{lem}

\begin{proof}
  Let us say that an orbit $\orbit (1,i)$ with $i\in X\setminus \{1\}$
  has weight
  $\frac{s-2}{2s}$, where $s$ is the size of the orbit. Then
  the left hand side of \eqref{eq:QRcondition} is just the weight sum
  of the orbits $\orbit (1,i)$ of size at least $3$, where $i\in X\setminus
  \{1\}$. By assumption there are $6$ such orbits, and the smallest weight is
  $1/6$ which appears if the orbit size is $3$.
  Hence \eqref{eq:QRcondition} implies that all weights are $1/6$, that
  is, all orbits have size $3$. This proves the claim.
\end{proof}

\textit{Step 1.} $\varphi_{1}=(2\,3\,4\,5\,6\,7)$.
By Lemma \ref{lem:solutions_m=6} and 
Corollary~\ref{cor:orbit3}
we obtain that $\varphi_2=(3\,1\,7\cdots)$, $\varphi_3=(4\,1\,2\cdots)$, 
$\varphi_4=(5\,1\,3\cdots)$, $\varphi_5=(6\,1\,4\cdots)$,
$\varphi_6=(7\,1\,5\cdots)$
and $\varphi_7=(2\,1\,6\cdots)$.
Since
\[
7=2\trid1=2\trid(3\trid4)=(2\trid3)\trid(2\trid4)=1\trid(2\trid4),
\]
it follows that $2\trid4=6$.
Moreover, $2\trid5\ne5$. Indeed, otherwise
\[
7=2\trid1=2\trid(4\trid5)=(2\trid4)\trid(2\trid5)=6\trid5\ne7,
\]
a contradiction.
Therefore $\varphi _2\in \{(3\,1\,7\,4\,6\,5),(3\,1\,7\,5\,4\,6)\}$.
By conjugation with $\varphi _1$ one obtains all permutations $\varphi _i$
with $i\in \{3,4,5,6,7\}$. Since $X$ is indecomposable,
Lemma~\ref{lem:cs_indecomposable} with $Y_1=Y_2=\{1,2,\dots ,7\}$
implies that $d=7$.

\begin{caseenv}
\item Assume that $\varphi_{2}=(3\,1\,7\,4\,6\,5)$. Then
$\varphi_{3}=(4\,1\,2\,5\,7\,6)$,
$\varphi_{4}=(5\,1\,3\,6\,2\,7)$,
$\varphi_{5}=(6\,1\,4\,7\,3\,2)$,
$\varphi_{6}=(7\,1\,5\,2\,4\,3)$ and
$\varphi_{7}=(2\,1\,6\,3\,5\,4)$.
This rack is the affine rack $\mathrm{Aff}(7,5)$.

\item Assume that $\varphi_{2}=(3\,1\,7\,5\,4\,6)$. Then
$\varphi_{3}=(4\,1\,2\,6\,5\,7)$,
$\varphi_{4}=(5\,1\,3\,7\,6\,2)$,
$\varphi_{5}=(6\,1\,4\,2\,7\,3)$,
$\varphi_{6}=(7\,1\,5\,3\,2\,4)$ and
$\varphi_{7}=(2\,1\,6\,4\,3\,5)$.
This rack is the affine rack $\mathrm{Aff}(7,3)$.
\end{caseenv}

\textit{Step 2.} $\varphi_{1}=(2\,3\,4)(5\,6\,7)$.
By Lemma \ref{lem:solutions_m=6} and 
Corollary~\ref{cor:orbit3}
we obtain that
$\varphi_{2}=(3\,1\,4)(\cdots)$, $\varphi_{3}=(4\,1\,2)(\cdots)$, 
$\varphi_{4}=(2\,1\,3)(\cdots)$,
$\varphi_{5}=(6\,1\,7)(\cdots)$, $\varphi_{6}=(7\,1\,5)(\cdots)$ and
$\varphi_7=(5\,1\,6)(\cdots)$.
Now we prove two lemmas and prove that they contradict to
$\varphi_{1}=(2\,3\,4)(5\,6\,7)$.
Let $x,y,z\in X$ such that $\varphi_{2}=(3\,1\,4)(x\, y\, z)$.

\begin{lem}
\label{lem:auxiliar1_m=6}
If $1\trid y=y$ then $3\trid x=x$. Moreover, $1\trid x\ne x$.
\end{lem}

\begin{proof}
From $2\trid(3\trid x)=(2\trid3)\trid(2\trid x)=1\trid y=y$ and from (R1)
we conclude that
$3\trid x=x$.
If $1\trid x=x$ then
$x=\varphi _1^2(3\trid x)=\varphi _1^2(3)\trid \varphi _1^2(x)=2\trid x=y$,
a contradiction.
\end{proof}

\begin{lem}
\label{lem:auxiliar2_m=6}
We have $2\trid5\ne6$ and $2\trid 5\not=7$.
\end{lem}

\begin{proof}
  If $2\trid 5=6$ then $5\trid 6=2$ by Corollary~\ref{cor:orbit3} and
  Lemma~\ref{lem:solutions_m=6}. This is a contradiction to $5\trid 6=1$.
  Similarly, if $2\trid 5=7$ then $5\trid 7=2$, a contradiction to $5\trid
  7=6$.
\end{proof}

Lemma~\ref{lem:auxiliar1_m=6} implies that at least two of $x,y,z$
are contained in $\{5,6,7\}$. By cyclic permutation of $x,y,z$ and of $5,6,7$
we may assume that $x,y\in \{5,6,7\}$ and that $x=5$.
Then Lemma~\ref{lem:auxiliar2_m=6} gives a contradiction to $2\trid x=y$.

We conclude that there are no racks satisfying the properties assumed in
Step~2.

\textit{Step 3.} $\varphi_{1}=(2\,3\,4\,5)(6\,7)$.
Lemma \ref{lem:solutions_m=6} and Corollary~\ref{cor:orbit3}
imply that
$\varphi_{2}=(3\,1\,5\,\cdot)(\cdot\,\cdot)$,
$\varphi_{3}=(4\,1\,2\,\cdot)(\cdot\,\cdot)$,
$\varphi_{4}=(5\,1\,3\,\cdot)(\cdot\,\cdot)$,
$\varphi_{5}=(2\,1\,4\,\cdot)(\cdot\,\cdot)$,
$\varphi_{6}=(1\,7)(\cdot\,\cdot\,\cdot\,\cdot)$
and $\varphi_{7}=(1\,6)(\cdot\,\cdot\,\cdot\,\cdot)$.
Since the role of $6$ and $7$ is exchangeable,
we may assume that $2\trid 5\in\{4,6,8\}$.
However, $2\trid 5\not=4$. Indeed, otherwise
$5\trid 4=2$ by Corollary~\ref{cor:orbit3}, which is a contradiction to 
$\varphi_{5}=(2\,1\,4\,\cdot)(\cdot\cdot)$.
Further, $2\trid 5\not=6$. Indeed, otherwise 
$\varphi_{2}=(3\,1\,5\,6)(\cdot\,\cdot)$ and
$5\trid 6=2$ by Corollary~\ref{cor:orbit3}. By applying $\varphi _1$
to the last equation
we obtain that $2\trid 7=3$ which is a contradiction to $2\trid 6=3$.
It follows that $2\trid 5=8$. By applying $\varphi _1^3$ we conclude that
$5\trid 4=8$.

Since 
$8=1\trid8=(5\trid2)\trid(5\trid4)=5\trid(2\trid4)$
and $5\trid4=8$, we obtain from (R1) that $2\trid4=4$.
Thus $\varphi_{2}\in \{(3\,1\,5\,8)(6\,7),
(3\,1\,5\,8)(6\,9),
(3\,1\,5\,8)(9\,10)\}$.

\textit{Step 3.1.} $\varphi_{2}=(3\,1\,5\,8)(6\,7)$.
Since $2\trid 6=7$, Corollary~\ref{cor:orbit3} implies that
$6\trid 7=2$, a contradiction.

\textit{Step 3.2.} $\varphi_{2}=(3\,1\,5\,8)(6\,9)$.
Then $6\trid 9=2$ by Corollary~\ref{cor:orbit3}. By applying $\varphi _1^2$
we obtain that $6\trid 9=4$, a contradiction.

\textit{Step 3.3.} $\varphi_{2}=(3\,1\,5\,8)(9\,10)$.
Then $9\trid 10=2$ by Corollary~\ref{cor:orbit3}. By applying $\varphi _1$
we obtain that $9\trid 10=3$, a contradiction.

It follows that there is no rack $X$ such that $k_3=6$ and $\varphi
_1=(2\,3\,4\,5)(6\,7)$.

\vspace{\baselineskip}

\textit{Step 4.} $\varphi_{1}=(2\,3)(4\,5)(6\,7)$.
Lemma \ref{lem:solutions_m=6} implies that
\begin{align*}
  \orbit (1,2)=\obl 3~1~2\obr ,\quad
  \orbit (1,3)=\obl 2~1~3\obr ,\quad
  \orbit (1,4)=\obl 5~1~4\obr ,\\
  \orbit (1,5)=\obl 4~1~5\obr ,\quad
  \orbit (1,6)=\obl 7~1~6\obr ,\quad
  \orbit (1,7)=\obl 6~1~7\obr .
\end{align*}
Hence, since $4\trid 5\not=2$, it follows from Corollary~\ref{cor:orbit3}
that $2\trid 4\not=5$. Similarly, $2\trid 6\not=7$. Further, $2\trid 8\not=9$
since otherwise
\[ 9=1\trid 9=\varphi _2\varphi _1(2)\trid \varphi _2\varphi _1(8)
   =\varphi _2\varphi _1(2\trid 8)=\varphi _2\varphi _1(9)=8,
\]
a contradiction. Hence, using our freedom to rename $4$, $5$, $6$, and $7$, we
may assume that
\[ \varphi _2\in \{ (1\,3)(4\,6)(5\,7),
   (1\,3)(4\,6)(5\,8),
   (1\,3)(4\,8)(5\,9),
   (1\,3)(4\,8)(6\,9) \}.
\]

\textit{Step 4.1.} $\varphi_2=(1\,3)(4\,6)(5\,7)$.
By conjugation with $\varphi _1$ we obtain that
$\varphi_3=(1\,2)(5\,7)(4\,6)$. Since $\#\orbit(2,4)=3$, it follows from
Corollary~\ref{cor:orbit3} that
$6\trid2=4$. Similarly, $\#\orbit(3,6)=3$ implies that $6\trid4=3$. 
This is a contradiction since $\varphi_6^2=\mathrm{id}$.

\textit{Step 4.2.} $\varphi_2=(1\,3)(4\,6)(5\,8)$.
Conjugation by $\varphi _1$ yields that
$\varphi_3=(1\,2)(5\,7)(4\,8)$, and then $\varphi _1=\varphi _{2\trid 3}
=(3\,2)(8\,7)(6\,5)$, a contradiction.

\textit{Step 4.3.} $\varphi_2=(1\,3)(4\,8)(5\,9)$.
As in the previous step we obtain that $\varphi _3=(1\,2)(5\,8)(4\,9)$
and $\varphi _1=\varphi _{2\trid 3}=(3\,2)(9\,4)(8\,5)$, a contradiction.

\textit{Step 4.4.} $\varphi_2=(1\,3)(4\,8)(6\,9)$.
Then conjugation by $\varphi _1$ yields that
$\varphi_3=(1\,2)(5\,8)(7\,9)$. Corollary~\ref{cor:orbit3} implies that
$\varphi_4=(1\,5)(2\,8)(\cdot\,\cdot )$
and
$\varphi_6=(1\,7)(2\,9)(\cdot\,\cdot )$.
Since $1\trid(8\trid9)=8\trid9$ by (R2) and $8\trid 9\not=1$ by
Corollary~\ref{cor:orbit3}, it follows that $8\trid9\in\{9,10\}$.

First we claim that $8\trid 9\not=9$. Indeed, otherwise $2\trid (8\trid
9)=2\trid 9$, and hence $4\trid 6=6$ by (R2). However, $\varphi _6\not=
 (5\,4\trid 7)(8\,4\trid 9)(\cdot \,\cdot )
 =\varphi _4\varphi _6\varphi _4^{-1}$, a contradiction.

The above arguments yield that $8\trid 9=10$. By conjugating with $\varphi _2$
we conclude that $4\trid 6=10$. It follows from Corollary~\ref{cor:orbit3}
and from (R2) that
\begin{align*}
  \varphi_1=&(2\,3)(4\,5)(6\,7),&
  \varphi_2=&(1\,3)(4\,8)(6\,9),&
  \varphi_3=&(1\,2)(5\,8)(7\,9),\\
  \varphi_4=&(1\,5)(2\,8)(6\,10),&
  \varphi_5=&(1\,4)(3\,8)(7\,10),&
  \varphi_6=&(1\,7)(2\,9)(4\,10),\\
  \varphi_7=&(1\,6)(3\,9)(5\,10),&
  \varphi_8=&(2\,4)(3\,5)(9\,10),&
  \varphi_9=&(2\,6)(3\,7)(8\,10),\\
  & & \varphi_{10}=&(4\,6)(5\,7)(8\,9).
\end{align*}
This rack is 
the rack $\mathcal{C}$.

\section{Yetter-Drinfeld modules}
\label{YD}

We refer to \cite{MR1913436} for an introduction to Yetter-Drinfeld modules
and Nichols algebras.

Let $G$ be a group. Let $g\in G$ and assume that the conjugacy
class $X$ of $g$ is finite and generates $G$.
Let $\fie $ be a field and let $V$ be a Yetter-Drinfeld module over $\fie G$.
Let $\delta:V\to \fie G\otimes V$ be the left coaction of $\fie G$ on $V$.
Then $V=\oplus_{g\in G}V_g$, where 
$V_g=\{v\in V\mid \delta{v}=g\otimes v\}$.
Moreover, $h V_g = V_{hgh^{-1}}$ for all $g,h\in G$.
Yetter-Drinfeld modules can also be studied in terms of racks and two-cocycles,
see \cite[Thm.\,4.14]{MR1994219}.

For any $g\in G$ and any representation $(\rho,W)$ of $C_{G}(g)$ let 
$M(g,\rho)=\mathrm{Ind}_{C_G(g)}^{G}\rho$ be the induced $G$-module. Then
$M(g,\rho)$ is a Yetter-Drinfeld module over $G$. The coaction of $\fie G$ 
on $M(g,\rho)$ is given by $\delta(h\otimes w)=hgh^{-1}\otimes(h\otimes w)$
for all $h\in G$ and $w\in W$.

We write $\NA(V)$ for the Nichols algebra of $V$ and $\NA_n(V)$ for the 
subspace of homogeneous elements of degree $n$, where $n\in\ndN$.

For all $g\in G$ and any linear functional $f\in V_g^*$ there exists a unique skew-derivation 
$\partial_f$ of the Nichols algebra $\NA(V)$ of $V$ such that 
\begin{align*}
\partial_f(v)&=f(v)&&\text{for all $v\in V_g$},\\
\partial_f(v)&=0&&\text{for all $v\in V_h$ with $h\in G\setminus\{g\}$},\\
\partial_f(xy)&=x\partial_f(y)+\partial_f(x)(gy)&&\text{for all $x,y\in \NA(V)$},
\end{align*}
see \cite[Proof of Lemma 3.5]{MR1779599}. 
If $\dim V_g=1$ for some $g\in G$ then we write $\partial_v$ for $\partial_{v^*}$, 
where $v\in V_g$ and $v^*$ is the dual basis vector of $v$.

\subsection{Relationship between Nichols algebras over different fields}

\begin{lem}
\label{lem:same_char}
Let $K$ be a field extension of $\fie$ and let $V_K=K\otimes_\fie V$. 
Then $V_K$ is a Yetter-Drinfeld module over $KG$ and 
any basis of $\NA(V)$ as a vector space over $\fie$ 
is a basis of $\NA(V_K)$ as a vector space over $K$. In particular, 
$\dim_K\NA(V_K)=\dim_\fie\NA(V)$.
\end{lem}

\begin{proof}
It is clear that any basis of $\NA(V)$ is spanning $\NA(V_K)$ as a vector space over $K$.
The linear independence can be obtained from the description of the Nichols algebra in 
terms of the quantum symmetrizer.
\end{proof}

\begin{rem}
\label{rem:tau_p}
Assume that $\mathrm{char}\,\fie=0$ and that  
$V=M(g,\rho)$ for some $g\in G$ and a one-dimensional representation $\rho$ of $C_G(g)$ 
such that $\rho(h)\in\{-1,1\}$ for all $h\in C_G(g)$. Let $d=\dim V$ and 
let $p$ be a prime number. Fix $v\in V_g\setminus\{0\}$. Then there exist $g_1,\dots,g_d\in G$ such that 
$\{g_1v,\dots,g_dv\}$ is a basis of $V$. 
Let 
\begin{gather}
\label{eq:tau_p_V}
\tau_pV=\mathrm{span}_\ndZ\{g_1v,\dots,g_dv\}/\mathrm{span}_\ndZ\{pg_1v,\dots,pg_dv\}.
\end{gather}
Then $\tau_pV$ is a Yetter-Drinfeld module over $\mathbb{F}_pG$. Up to isomorphism 
the definition of $\tau_pV$ does not
depend on the choices of $g_1,\dots,g_d$ and $v$.
\end{rem}

The following result was proved in \cite{MR1994219} for $\fie=\ndC$. 
The proof of the theorem holds without any restriction on $\fie$.

\begin{thm}
\label{thm:AG}\cite[Theorem 6.4]{MR1994219}
Assume that $\NA(V)$ is finite-dimensional. Let $m\in\ndN$ such that $\dim\NA_m(V)=1$ and 
$\dim\NA_n(V)=0$ for all $n>m$. Let $\mathcal{J}\subseteq TV$ be an $\ndN$-graded Yetter-Drinfeld 
submodule over $\fie G$ such that $\mathcal{J}\cap \fie =\mathcal{J}\cap V=0$.
Assume 
that $\mathcal{J}$ is an ideal and a coideal of $TV$ such that $\dim T^mV/(\mathcal{J}\cap T^mV)=1$ 
and $\dim T^nV/(\mathcal{J}\cap T^nV)=0$ for all
$n>m$. Then $TV/\mathcal{J}\simeq\NA(V)$.
\end{thm}

\begin{thm}
\label{thm:diff_char}
Assume that $\fie=\ndQ$ and that  
$V=M(g,\rho)$ for some $g\in G$ and a one-dimensional representation $\rho$ of $C_G(g)$ 
such that $\rho(h)\in\{-1,1\}$ for all $h\in C_G(g)$.
Let $p$ be a prime number. 
Then the following hold.
\begin{enumerate}
\item $\dim_{\mathbb{F}_p}\NA_n(\tau_pV)\leq\dim_\fie\NA_n(V)$ for all $n\in\ndN$.
\item
  Let $m\in\ndN$ such that $\dim_\fie\NA_m(V)=1$ and $\dim_\fie\NA_n(V)=0$
  for all $n>m$. If $\dim_{\mathbb{F}_p}\NA_m(\tau_pV)\geq1$
  then $\dim_{\mathbb{F}_p}\NA_n(\tau_pV)=\dim_\fie\NA_n(V)$ for all $n\in\ndN$.
\end{enumerate}
\end{thm}

\begin{proof}
Let $d=\dim V$ and let $g_1,\dots,g_d\in G$ such that $\{g_1v,\dots,g_dv\}$ is a basis of $V$. 
We write $v_i$ for $g_iv$ for all $i\in\{1,\dots,d\}$. 
The elements $v_1,\dots,v_d\in V$ generate an additive subgroup of $V$ which is denoted by $W$. 
For all $n\in\ndN$ let $W^{\otimes_\ndZ n}=W\otimes_\ndZ W\otimes_\ndZ\cdots\otimes_\ndZ W$ and 
let $f_p:W^{\otimes_\ndZ n}\to (\tau_pV)^{\otimes_{\mathbb{F}_p} n}$ be the 
canonical group homomorphism induced by the canonical map $\ndZ\to\mathbb{F}_p$. 

First we prove $(1)$. Let $n\in\ndN$ and 
let $\mathcal{B}$ be a subset of 
the $n$-fold tensor product $W^{\otimes_\ndZ n}$ consisting of tensor products 
of generators $v_i$, $i\in\{1,\dots,d\}$.
Assume that $f_p(\mathcal{B})$ is linearly independent in $\NA_n(\tau_pV)$. It suffices to show that 
$\mathcal{B}$ is linearly independent in $\NA_n(V)$. We give an indirect proof. 

Let $r\in\ndN$, $b_1,\dots,b_r\in\mathcal{B}$ and
$\lambda_1,\dots,\lambda_r\in\ndQ$ such that 
$\sum_{i=1}^r \lambda_i b_i=0$ in $\NA_n(V)$ 
and that $\lambda_1\ne0$.
Without loss of generality we may assume that 
$\lambda_i\in\ndZ$ for all $i\in\{1,\dots,r\}$ 
and that $\lambda_1\notin p\ndZ$. The assumption on $\rho$ implies that $W^{\otimes_\ndZ n}$ 
is stable under the action of the quantum symmetrizer. Moreover, the map $f_p$ commutes with 
the quantum symmetrizer by construction. 
Then the characterization of Nichols algebras in terms of 
quantum symmetrizers yields that $f_p(\mathcal{B})$ is linearly 
dependent in $\NA_n(\tau_pV)$, which proves the claim.

Now we prove $(2)$. For all $n\in\ndN$ with $n\geq 2$ let $U^n\subseteq V^{\otimes n}$ 
be the kernel of the quantum symmetrizer and let $U^n_\ndZ=U^n\cap W^{\otimes_\ndZ n}$. 
Then $\oplus_{n\geq2}U^n_\ndZ/pU^n_\ndZ$ is a Yetter-Drinfeld submodule of the tensor algebra $T(\tau_pV)$ of $\tau_pV$ 
by the assumption on $\rho$. Moreover, it is an ideal and a coideal 
of $T(\tau_pV)$ since $U$ is an ideal and a coideal of $TV$. Thus $B=T(\tau_pV)/\oplus_{n\geq2}U^n_\ndZ/pU^n_\ndZ$ 
is an $\ndN$-graded braided bialgebra in the category of Yetter-Drinfeld modules over $\mathbb{F}_pG$. 
Lemma \ref{lem:ed} yields that 
\begin{gather}
\label{eq:hilbert}
\mathcal{H}_{B}(t)=\mathcal{H}_{\NA(V)}(t),
\end{gather}
where $\mathcal{H}$ is the Hilbert series.
By assumption we know that $\dim B_m=1$ and that $\dim B_n=0$ for all $n>m$.
By $(1)$ and the assumption on $\NA_m(\tau_pV)$
we conclude that $\dim\NA_m(\tau_pV)=1$ and $\dim\NA_n(\tau_pV)=0$ for all
$n>m$. By Theorem \ref{thm:AG} with $\fie=\mathbb{F}_p$ it follows that
$\NA(\tau_pV)\simeq B$. Thus the claim holds by Equation \eqref{eq:hilbert}.
\end{proof}

\subsection{Quadratic relations}

Let $g\in G$ and let $(\rho,W)$ be a representation of $C_G(g)$.
We write $X$ for the conjugacy class of $g$ in $G$.
Assume that $V=M(g,\rho)$. Then $V_g\simeq W$ as $C_G(g)$-modules. Let $d$ be 
the number of elements of $X$ and let 
$e=\dim V_g$.  For any $H$-orbit $\orbit$ in $X\times X$, see Subsection \ref{ss:braided}, let 
\begin{gather}
\label{eq:VO2}
V_\orbit^{\otimes2}=\mathop{\oplus}_{(x,y)\in\orbit}V_x\otimes V_y.
\end{gather}
Then $V\otimes V=\oplus_{\orbit} V_\orbit^{\otimes2}$, where $\orbit$ is running over 
all $H$-orbits in $X\times X$. 

\begin{lem}
\label{lem:kernel}
Let $r\in V\otimes V$. For any $H$-orbit $\orbit$
let $r_\orbit$ be the projection of $r$ to $V_\orbit^{\otimes2}$.
If $(1+c)(r)=0$ then $(1+c)(r_\orbit)=0$ for all $\orbit$. 
\end{lem}

\begin{proof}
By construction, $(1+c)(r_\orbit)\in V_\orbit^{\otimes2}$ for any $H$-orbit $\orbit$. 
This implies the claim.
\end{proof}

Assume that $V$ is finite-dimensional and 
absolutely irreducible, that is, $K\otimes _\fie V$ is an irreducible
Yetter-Drinfeld module over $KG$ for any field extension $K$ of $\fie $.
In this case the Lemma of Schur implies that 
any central element in $C_G(h)$, where $h\in X$, acts on $V_h$ 
by a scalar. 

\begin{rem}
\label{rem:YD_ai}
A Yetter-Drinfeld module over $G$ is absolutely irreducible 
if and only if it is isomorphic to $M(h,\sigma)$, where $h\in G$
and $\sigma$ is an absolutely irreducible module of the centralizer $C_G(h)$.
\end{rem}

\begin{lem}
\label{lem:bound1}
Let $h\in X$ and let $\orbit=\orbit(h,h)$. Then
\[
\dim\big(\ker(1+c)\cap V_\orbit^{\otimes2}\big)\leq\frac{e(e+1)}{2}.
\]
\end{lem}

\begin{proof}
Let $v_{1},\dots ,v_{e}$ be a basis of $V_{h}$. 
Let $\lambda\in\fie$ such that $hv=\lambda v$ for all $v\in V_h$.
Then
\[
(1+c)(v_{j}\otimes v_{k})=v_{j}\otimes v_{k}+hv_{k}\otimes v_{j}=v_j\otimes v_{k}+\lambda v_{k}\otimes v_{j}
\] 
for all $j,k\in\{1,\dots,e\}$ with $j<k$. 
Hence $\lspan\{v_{j}\otimes v_{k}\mid j<k\}$ 
has trivial intersection with $\ker(1+c)$.
This completes the proof.
\end{proof}

\begin{lem}
\label{lem:bound2}
Let $h_1,h_2\in X$ with $h_1\ne h_2$ and let $\orbit=\orbit(h_1,h_2)$. Then
\[
\dim\big(\ker(1+c)\cap V_\orbit^{\otimes2}\big)\leq e^2.
\]
\end{lem}

\begin{proof}
Let $r\in\ker(1+c)\cap V_\orbit^{\otimes2}$ and for all $(g_1,g_2)\in\orbit$ 
let $r_{(g_1,g_2)}$ be the projection
of $r$ to $V_{g_1}\otimes V_{g_2}$.
Since $h_1\ne h_2$, we conclude that $c(g_1,g_2)\ne (g_1,g_2)$ 
for all $(g_1,g_2)\in\orbit$. 
Since $r\in\ker(1+c)$, it follows that $c(r_{(g_1,g_2)})=-r_{c(g_1,g_2)}$
for all $(g_1,g_2)\in\orbit$.
Thus the transitivity of the action of $H$ on $\orbit$
implies that $r$ is uniquely determined by $r_{(h_1,h_2)}$.
This completes the proof.
\end{proof}

Next we demostrate the calculation of quadratic relations on an example. 
This result is related to \cite[Lemma 2.2]{GG2009}. We prepare 
the example with the following lemma. 

\begin{lem}
\label{lem:qr}
Let $n\in\ndN$ and let $x,y\in X$, $v\in V_x$ and $w\in V_y$. Then 
\[
c^n(v\otimes w)=
\begin{cases}
(xy)^kx^{-k}v\otimes (xy)^ky^{-k}w & \text{ if $n=2k$, $k\in\ndN$}\\
(xy)^{k+1}y^{-k-1}w\otimes (xy)^kx^{-k}v & \text{ if $n=2k+1$, $k\in\ndN$}.
\end{cases}
\]
\end{lem}

\begin{proof}
By induction on $n$.
\end{proof}

\begin{example}
\label{exa:qr}
Assume that $G=G_X$ and that $\dim V_g=1$. 

Let $x\in X$ and let $\orbit=\orbit(x,x)$. 
Then $\dim(V_x\otimes V_x)=1$, and $(1+c)(V_{\orbit}^{\otimes2})=0$ if and only if
$\rho(g)=-1$.

Let $x,y\in X$ with $x\ne y$. Let $\orbit=\orbit(x,y)$. The proof of Lemma \ref{lem:bound2} 
gives that $\dim(\ker(1+c)\cap V_\orbit^{\otimes2})\leq 1$. More precisely, 
let $r\in V_\orbit^{\otimes2}$ and for all $(g_1,g_2)\in\orbit$ let 
$r_{(g_1,g_2)}$ be the projection of $r$ to $V_{g_1}\otimes V_{g_2}$. 
Then $r\in\ker(1+c)$ if and only if $r_{c^n(x,y)}=(-1)^n c^n(r_{(x,y)})$ for all $n\in\ndN$. 
Hence $\dim(\ker(1+c)\cap V_\orbit^{\otimes2})=1$ if and only if 
\begin{gather}
\label{eq:cn}
\dim\big(\ker(c^m-(-1)^m)\cap V_x\otimes V_y\big)=1, 
\end{gather}
where $m=\#\orbit$. Assume now that $x=g$ and $y=hgh^{-1}$,
where $h\in G_X$ such that $y=hxh^{-1}$. Then
Lemma \ref{lem:qr} implies that
Equation \eqref{eq:cn} is equivalent to
\begin{gather}
\label{eq:rho}
\rho\big(\mathrm{Prod}(hg,h^{-1}g;m)\,\mathrm{Prod}(h^{-1}g,hg;m)\big)
=(-1)^m\rho(g)^m.
\end{gather}
Assuming the latter one further obtains that
\[
r=\lambda\sum_{n=0}^{m-1}(-1)^n c^n(v\otimes hv),
\]
where $\lambda\in\fie$ and $v\in V_g\setminus\{0\}$.
\end{example}

Recall the definition of $k_n$ and $l_n$, where $n\in\ndN$, from Subsection \ref{ss:braided}.

\begin{prop}
\label{prop:number}
The dimension of $\ker(1+c)$ is at most 
\begin{gather}
\label{eq:number}
d\frac{e(e+1)}{2}+l_{2}e^{2}+l_{3}e^{2}+\cdots
\end{gather}
\end{prop}

\begin{proof}
Follows from Lemmas \ref{lem:kernel}, \ref{lem:bound1} and \ref{lem:bound2}.
\end{proof}

\begin{rem}
\label{rem:lk}
Lemma \ref{lem:l_n} implies that 
\[
d\frac{e(e+1)}{2}+l_{2}e^{2}+l_{3}e^{2}+\cdots=e\left(\frac{d(e+1)}{2}+\frac{dk_{2}}{2}e+\frac{dk_{3}}{3}e+\cdots\right).
\]
\end{rem}

\begin{cor}
\label{cor:YD_condition}
Assume that $\dim\ker(1+c)\geq \dim V(\dim V-1)/2$. Then
\begin{equation}
\label{eq:YD_condition}
\sum_{n\geq3}\frac{n-2}{2n}k_{n}\leq\frac{1}{e}.
\end{equation}
\end{cor}

\begin{proof}
Recall that $\dim V=de$. 
Since $\dim\ker(1+c)\geq de(de-1)/2$, Proposition \ref{prop:number} and Remark \ref{rem:lk} imply that 
\[
\frac{1}{2}de(de-1)\leq e\left(\frac{d(e+1)}{2}+\frac{dk_{2}}{2}e+\frac{dk_{3}}{3}e+\cdots\right).
\]
Lemma \ref{lem:general}\eqref{lem:counting_formula} applied to the left hand side of this inequality 
yields the claim.
\end{proof}



For all $n\in \ndN $ let $(n)_t=1+t+t^2+\cdots +t^{n-1}$.

\begin{thm}
\label{thm:YD}
Let $G$ be a group, $\fie$ a field, and $V$ a Yetter-Drinfeld module over the group algebra $\fie G$. 
Assume that $V$ is finite-dimensional and absolutely irreducible. 
Then $V\simeq M(g,\rho)$ for some $g\in G$ and an absolutely simple representation $\rho$ of $C_G(g)$.
Let $d_V=\dim V$.
The following are equivalent.
\begin{enumerate}
\item $\dim\mathfrak{B}_2(V)\leq\frac{d_V(d_V+1)}{2}$.
\item $\dim V_g=1$, the conjugacy class of $g$ 
is isomorphic as a rack to one of the racks listed in Table \ref{tab:racks}, and
the representation of $C_G(g)$ is given in Table \ref{tab:candc}, where 
$p:G_X\to G$ is the canonical projection and $g=px_1$.
\item There exist $n_1,n_2,\ldots ,n_{d_V}\in\ndN $
  such that the Hilbert series of $\NA(V)$ factorizes as
  $\mathcal{H}_{\mathfrak{B}(V)}(t)=(n_1)_t(n_2)_t\cdots(n_{d_V})_t$.
\end{enumerate}
\end{thm}

\begin{proof}
First we prove that $(1)$ implies $(2)$. Let $e=\dim V_g$. Since 
\[
\NA_2(V)\simeq V\otimes V/\ker(1+c),
\]
claim $(1)$ implies that $\dim\ker(1+c)\geq d_V(d_V-1)/{2}$.
By Corollary \ref{cor:YD_condition} we obtain that
\[
\sum_{n\geq3}\frac{n-2}{2n}k_{n}\leq\frac{1}{e}.
\]
If $e=1$ then by Theorem \ref{thm:racks} the conjugacy class of $g$ 
is isomorphic as a rack to one of the racks listed in Table \ref{tab:racks}.
The claim on the representation of $C_G(g)$ is proved case by case in 
Section \ref{s:examples}.

Assume that $e\geq 2$. Then from Theorem \ref{thm:racks} and the last column 
of Table \ref{tab:racks} we conclude that 
the conjugacy class of $g$ is isomorphic as a rack to $\mathbb{D}_3$
or $\mathcal{T}$. 
Let $p:G_X\to G$ be the canonical projection and let $x\in p^{-1}g$. 
Remark \ref{rem:centralizer} implies that $p^{-1}C_G(g)=C_{G_X}(x)$. 
The centralizer $C_{G_X}(x)$ is abelian 
by Lemmas \ref{lem:D3_cen} and \ref{lem:T_cen}, and hence $C_G(g)$ is abelian.
Since $V$ is absolutely irreducible, Remark \ref{rem:YD_ai} implies that 
$e=\dim V_g=1$, a contradiction.

The implication $(2)\Rightarrow(3)$ is proved case by case in Section \ref{s:examples}. 
 
Finally, $(3)$ implies $(1)$ since $\dim\NA_2(V)$ is the coefficient of $t^2$ in 
the Hilbert series $\mathcal{H}_{\NA(V)}(t)$.
\end{proof}

\begin{table}[ht]
\caption{Centralizers and characters}
\begin{center}
\label{tab:candc}
\begin{tabular}{|c|c|c|}
\hline 
Rack & Generators of $C_{G}(px_1)$ & Linear character $\rho$ on $C_{G}(px_1)$ \tabularnewline
\hline 
$\mathbb{D}_3$ & $px_1$ & $\rho(px_1)=-1$\\
\hline 
$\mathcal{T}$ & $px_1,\;p(x_4x_2)$ & $\rho(px_1)=-1,\;\rho(p(x_4x_2))=1$\\
\hline 
$\mathrm{Aff}(5,2)$ & $px_1$ & $\rho(px_1)=-1$\\
\hline 
$\mathrm{Aff}(5,3)$ & $px_1$ & $\rho(px_1)=-1$\\
\hline 
$\mathcal{A}$ & $px_1,\;px_4$& $\rho(px_1)=-1,\;\rho(px_4)=\pm1$\\
\hline 
$\mathcal{B}$ & $px_1,\;px_6$ & $\rho(px_1)=\rho(px_6)=-1$\\
\hline
$\mathrm{Aff}(7,3)$ & $px_1$ & $\rho(px_1)=-1$\\
\hline 
$\mathrm{Aff}(7,5)$ & $px_1$ & $\rho(px_1)=-1$\\
\hline 
$\mathcal{C}$ & $px_1,\;px_8,\;px_9$ & $\rho(px_1)=-1,\;\rho(px_8)=\rho(px_9)=\pm1$\\
\hline
\end{tabular}
\end{center}
\end{table}

\section{Examples}

\label{s:examples}
Let $\fie$ be a field and let $X$ be an indecomposable injective finite rack. 
Recall that $G_X$ is the enveloping group of $X$. 
Identify $X$ with $\{1,\dots,d\}$, where $d=\#X$, and write $x_i\in G_X$ for 
the image of $i\in X$ in $G_X$.

\begin{lem}
\label{lem:rho(g)=-1}
Let $\rho$ be a linear character of $C_{G_X}(x_1)$. 
Let $V=M(x_1,\rho)$ and $d_V=\dim V$. Assume that $d_V>1$. 
If $\dim\NA_2(V)\leq d_V(d_V+1)/2$ then
$\rho(x_1)=-1$. 
\end{lem}

\begin{proof}
Assume that $\rho(x_1)\ne-1$. Then $v^2\ne 0$ in $\NA(V)$ for all 
$v\in V_x\setminus\{0\}$ and all $x\in X$.
Since $\dim V_x=1$ for all $x\in X$, we conclude that
\[
\dim\NA_2(V)\geq d_V^2 - \#\{\orbit(x,y)\mid x,y\in X,\,x\ne y\}
\]
by Lemmas \ref{lem:kernel} and \ref{lem:bound2}. 
Since $d_V>1$ and $X$ is indecomposable, it follows that 
$\dim\NA_2(V)>d_V(d_V+1)/2$, a contradiction.
\end{proof}


\subsection{The rack $\mathbb{D}_3$}

Let $X=\mathbb{D}_3$. Then $\overline{G_X}\simeq\mathbb{S}_3$. 
In fact, there is an isomorphism
$\overline{G_X}\to\mathbb{S}_3$ given by
\[
\pi x_1\mapsto(2\,3),\quad \pi x_2\mapsto(1\,2),\quad \pi x_3\mapsto(1\,3).
\]

\begin{lem}
\label{lem:D3_cen}
The centralizer of $x_1$ in $G_X$ is the cyclic group generated by $x_1$.
\end{lem}

\begin{proof}
Follows from Lemma \ref{lem:centralizer}\eqref{lem:centralizer_cy}, since the centralizer of $\pi x_1$ in $\mathbb{S}_3$
is the cyclic group generated by $\pi x_1$.
\end{proof}

\begin{rem}
\label{rem:D3}
Since $G_{G_X}(x_1)$ is abelian, any absolutely simple representation of $G_{G_X}(x_1)$ 
is a linear character. 
Since $C_{G_X}(x_1)$ is cyclic, any linear character on $C_{G_X}(x_1)$ is determined by
its action on $x_1$. 
\end{rem}

\begin{prop}
\label{pro:D3_main}
Let $\rho$ be an absolutely simple representation of $C_{G_X}(x_1)$. Let $V=M(x_1,\rho)$ 
and $d_V=\dim V$.
\begin{enumerate}
\item The representation $\rho$ is a linear character on $C_{G_X}(x_1)$ and hence $d_V=3$. 
Moreover, $\dim\NA_2(V)\leq d_V(d_V+1)/2$ if and only if $\rho(x_1)=-1$. 
\item Assume that $\rho(x_1)=-1$. Then the following hold.
\begin{enumerate}
\item $\mathcal{H}_{\NA(V)}(t)=(2)^2_t (3)_t$ and $\dim\NA(V)=12$.
\item Let $v_i\in V_{x_i}$ with $i\in X$ be non-zero elements. Then $v_1v_2v_1v_3$ is an integral of $\NA(V)$.
\end{enumerate}
\end{enumerate}
\end{prop}

\begin{proof}
The representation $\rho$ is a linear character on $C_{G_X}(x_1)$ by Remark \ref{rem:D3}. 
Further, if $\dim\NA_2(V)\leq d_V(d_V+1)/2$ then $\rho(x_1)=-1$ by Lemma \ref{lem:rho(g)=-1}.

Assume now that $\rho(x_1)=-1$. It suffices to prove part $(2)$ of the claim. This follows 
from \cite[Example 6.4]{MR1800714}. 
\end{proof}

\subsection{The rack $\mathcal{T}$}

Let $X=\mathcal{T}$. Recall that in $G_X$ the following relations hold.
\begin{align}
\label{eq:GT_relations1}
x_1x_2&=x_4x_1=x_2x_4,& x_1x_4&=x_3x_1=x_4x_3,\\
\label{eq:GT_relations2}
x_1x_3&=x_2x_1=x_3x_2,& x_2x_3&=x_4x_2=x_3x_4.
\end{align}

\begin{lem}
\label{lem:T_cen}
The centralizer of $x_1$ in $G_X$ is the abelian group 
generated by $x_1$ and $x_4 x_2$.
These elements satisfy the relation $(x_4 x_2)^2=x_1^4$.
\end{lem}

\begin{proof}
First one checks that $x_4 x_2\in C_{G_X}(x_1)$ and hence $C=\langle x_4x_2,x_1\rangle$ 
is a subgroup of $C_{G_X}(x_1)$.
Lemmas \ref{lem:central} and \ref{lem:powers} imply that for all $i\in X$ 
the elements $x_i^3\in G_X$ are central and that $x_1^3=x_2^3=x_3^3=x_4^3$. 
By Equations \eqref{eq:GT_relations1} and \eqref{eq:GT_relations2} we conclude that 
\[
G_X/C=\{C,x_2C,x_3C,x_4C\}.
\]
Indeed, for example 
\[
x_2^2C=x_2^2x_1C=x_2x_3x_2C=x_3x_4x_2C=x_3C.
\]
Hence $[G_X:C]\leq 4$. 
Since $\#\mathcal{O}_{x_1}=4$, it follows that $C_{G_X}(x_1)=\langle x_1,x_4x_2\rangle$. 
Moreover, 
\[
(x_4 x_2)^2=x_4(x_2x_4)x_2=x_4^2(x_1x_2)=x_4^3x_1
\]
and Lemma \ref{lem:powers} implies that $x_4^3x_1=x_1^4$. This concludes the proof.
\end{proof}

\begin{prop}
\label{pro:T_main}
Let $\rho$ be an absolutely simple representation of $C_{G_X}(x_1)$. Let $V=M(x_1,\rho)$ 
and $d_V=\dim V$.
\begin{enumerate}
\item The representation $\rho$ is a linear character on $C_{G_X}(x_1)$, hence $d_V=4$. 
Moreover, $\dim\NA_2(V)\leq d_V(d_V+1)/2$ if and only if $\rho(x_1)=-1$ and $\rho(x_4x_2)=1$. 
\item Assume that $\mathrm{char}\,\fie\ne2$, $\rho(x_1)=-1$ and $\rho(x_4x_2)=1$. Then the following hold.
\begin{enumerate}
\item $\mathcal{H}_{\NA(V)}(t)=(2)^2_t (3)_t(6)_t$ and $\dim\NA(V)=72$.
\item Let $v_i\in V_{x_i}$ with $i\in X$ be non-zero elements. Then $v_1v_2v_1v_3v_2v_1v_3v_2v_4$ is an integral of $\NA(V)$.
\end{enumerate}
\item Assume that $\mathrm{char}\,\fie=2$, $\rho(x_1)=-1$ and $\rho(x_4x_2)=1$. Then the following hold.
\begin{enumerate}
\item $\mathcal{H}_{\NA(V)}(t)=(2)^2_t (3)^2_t$ and $\dim\NA(V)=36$.
\item Let $v_i\in V_{x_i}$ with $i\in X$ be non-zero elements. Then $v_1v_2v_1v_3v_2v_4$ is an integral of $\NA(V)$.
\end{enumerate}
\end{enumerate}
\end{prop}

\begin{proof}
The representation $\rho$ is a linear character on $C_{G_X}(x_1)$, since $C_{G_X}(x_1)$ is abelian. 
Further, if $\dim\NA_2(V)\leq d_V(d_V+1)/2$ then $\rho(x_1)=-1$ by Lemma \ref{lem:rho(g)=-1}. 
We prove that if $\dim\NA_2(V)\leq d_V(d_V+1)/2$ then $\rho(x_4x_2)=1$. Let $g=x_1$ and $h=x_2$. 
Then $hgh^{-1}=x_4$. 
Since $\#\orbit(x_1,x_4)=3$, Example \ref{exa:qr} implies 
that $\dim\big(\ker(1+c)\cap V_{\orbit(x_1,x_4)}^{\otimes2}\big)=1$ if and only if 
\[
\rho\big((x_2x_1x_2^{-1}x_1x_2x_1)(x_2^{-1}x_1x_2x_1x_2^{-1}x_1)\big)=-\rho(x_1)^3.
\]
By Equations \eqref{eq:GT_relations1}--\eqref{eq:GT_relations2} this is equivalent to $\rho(x_4x_2)\rho(x_1)=-1$. 
The group $G_X$ acts transitively on the set of orbits $\orbit(x,y)$ with $x\ne y$ by diagonal action. Hence 
$\dim\big(\ker(1+c)\cap V_{\orbit(x,y)}^{\otimes2}\big)=1$ for $x\ne y$ if and only if $\rho(x_4x_2)\rho(x_1)=-1$. 
The inequality $\dim\NA_2(V)\leq d_V(d_V+1)/2$ is equivalent 
to $\dim\ker(1+c)\geq 6$. Since there are four orbits $\orbit(x,y)$ with $x=y$ and
four orbits with $x\ne y$, the inequality $\dim\NA_2(V)\leq d_V(d_V+1)/2$ implies that 
$\rho(x_1)=-1$ and $\rho(x_4x_2)=1$. 
The remaining implication in part $(1)$ follows from $(2)$ and $(3)$. 

Now we prove $(2)$. For $\fie=\ndC$ the result is known, see \cite[Theorem 6.15]{MR1994219}. For 
other fields $\fie$ of characteristic zero the result follows from Lemma \ref{lem:same_char}. 
Assume now that $\mathrm{char}\,\fie >2$.
By the proof of \cite[Theorem 6.15]{MR1994219} the monomial $v_1v_2v_1v_3v_2v_1v_3v_2v_4$ 
is non-zero in $\NA(V)$. Therefore the claim in $(2)$ 
holds by Theorem \ref{thm:diff_char} for $\fie=\mathbb{F}_p$ with $p>2$ and by Lemma \ref{lem:same_char} 
for arbitrary fields $\fie$ of characteristic $p>2$.

The claims in $(3)$ follow from Proposition \ref{pro:Tchar2}.
\end{proof}

\begin{prop}
\label{pro:Tchar2}
Assume that $\mathrm{char}\,\fie=2$. 
\begin{enumerate}
\item Let $B$ be the algebra 
given by generators $a,b,c,d$ and defining relations
\begin{gather}
a^{2}=b^{2}=c^{2}=d^{2}=0\label{eq:T_rel1},\\
ba+db+ad=ca+bc+ab=da+cd+ac=cb+dc+bd=0\label{eq:T_rel2},\\
cad+bac+dab=0\label{eq:T_rel4_NEW}.
\end{gather}
Then a basis of $B$ is given by
\begin{gather*}
1,a,b,c,d,ab,ac,ad,ba,bc,bd,cb,cd,\\
aba,abc,abd,acb,acd,bac,bad,bcb,bcd,cbd,\\
abac,abad,abcb,abcd,acbd,bacb,bacd,bcbd,\\
abac,baba,cdab,cbdb,acbd,abacbd.\end{gather*}
\item 
Let $\rho$ be an absolutely simple representation of $C_{G_X}(x_1)$ such that 
$\rho(x_1)=-1$ and $\rho(x_4x_2)=1$. Let $V=M(x_1,\rho)$ and let 
$v\in V_{x_1}\setminus\{0\}$. Define 
\[
v_1=v,\quad v_2=x_3v,\quad v_3=x_4v,\quad v_4=x_2v. 
\]
Then the linear map $V\to B$ given by $(v_1,v_2,v_3,v_4)\mapsto (a,b,c,d)$ extends uniquely 
to an algebra isomorphism $\NA(V)\to B$. 
\end{enumerate}
\end{prop}

\begin{proof}
Part $(1)$ can be obtained using the diamond lemma \cite{MR506890} or by calculating a 
non-commutative Gr\"obner basis \cite{GBNP, GAP}. 
To prove part $(2)$ it is sufficient to see that $v_1v_2v_1v_3v_2v_4$ 
does not vanish in $\mathfrak{B}(V)$. 
For this purpose one can show that 
\[
\partial_{v_2}\partial_{v_1}\partial_{v_4}\partial_{v_1}\partial_{v_2}\partial_{v_3}(v_1v_2v_1v_3v_2v_4)=1.
\]
This completes the proof. 
\end{proof}

\subsection{The rack $\mathcal{A}$}

Let $X=\mathcal{A}$. Recall that in $G_X$ the following relations hold.
\begin{align}
\label{eq:GA_relations1}
x_1x_2&=x_3x_1=x_2x_3,& x_1x_3&=x_2x_1=x_3x_2,& x_1x_4&=x_4x_1,\\
\label{eq:GA_relations2}
x_1x_5&=x_6x_1=x_5x_6,& x_1x_6&=x_5x_1=x_6x_5, & x_2x_6&=x_6x_2,\\
\label{eq:GA_relations3}
x_2x_4&=x_5x_2=x_4x_5,& x_2x_5&=x_4x_2=x_5x_4, & x_3x_5&=x_5x_3,\\
\label{eq:GA_relations4}
x_3x_4&=x_6x_3=x_4x_6,& x_3x_6&=x_4x_3=x_6x_4.
\end{align}


\begin{lem}
\label{lem:A_cen}
The centralizer of $x_1$ in $G_X$ is the abelian group generated by $x_1$ and $x_4$.
These generators satisfy $x_1^2=x_4^2$.
\end{lem}

\begin{proof}
By \eqref{eq:GA_relations1} we obtain that $\langle x_1,x_4\rangle$ is an abelian subgroup of $C_{G_X}(x_1)$.
Equations \eqref{eq:GA_relations1}--\eqref{eq:GA_relations4} imply that $[G_X:\langle x_1,x_4\rangle]\leq 6$. 
Indeed, let $C=\langle x_1,x_4\rangle$. Then $x_2x_6x_1=x_3x_5x_4$. This and similar calculations yield that 
\[
G_X/C=\{C,x_3C, x_2C, x_5x_3C, x_6C, x_5C\}.
\]
Since $\#\mathcal{O}_{x_1}=6$, it follows that $C_{G_X}(x_1)=\langle x_1,x_4\rangle$. 
Equation $x_1^2=x_4^2$ is obtained from Lemma \ref{lem:powers}. 
\end{proof}

\begin{prop}
\label{pro:A_main}
Let $\rho$ be an absolutely simple representation of $C_{G_X}(x_1)$. Let $V=M(x_1,\rho)$ 
and $d_V=\dim V$.
\begin{enumerate}
\item The representation $\rho$ is a linear character on $C_{G_X}(x_1)$, hence $d_V=6$. 
Moreover, $\dim\NA_2(V)\leq d_V(d_V+1)/2$ if and only if $\rho(x_1)=-1$. In this case $\rho(x_4)\in\{-1,1\}$. 
\item Assume that $\rho(x_1)=-1$. Then the following hold.
\begin{enumerate}
\item $\mathcal{H}_{\NA(V)}(t)=(2)^2_t (3)^2_t (4)^2_t$ and $\dim\NA(V)=576$.
\item Let $v_i\in V_{x_i}$ with $i\in X$ be non-zero elements. Then the monomial 
$v_1v_2v_1v_3v_4v_2v_1v_3v_4v_5v_1v_6$ is an integral of $\NA(V)$.
\end{enumerate}
\end{enumerate}
\end{prop}

\begin{proof}
We first prove $(1)$. Since $C_{G_X}(x_1)$ is abelian by Lemma \ref{lem:A_cen}, we conclude that 
$\rho$ is a linear character. 
Lemma \ref{lem:rho(g)=-1} implies that $\rho(x_1)=-1$ and $\rho(x_4)\in\{-1,1\}$. 
The diagonal action of $G_X$ on the set of 
$H$-orbits of size $1$, $2$ and $3$, respectively, is transitive. 
Hence there is a quadratic relation of $\NA(V)$ corresponding to an $H$-orbit of a given size 
if and only if 
there is a quadratic relation of $\NA(V)$ for each $H$-orbit of this size. 
By Example \ref{exa:qr} the latter holds for all $H$-orbits.
The cases $\rho(x_4)=-1$ and $\rho(x_4)=1$, respectively, were discussed 
in \cite[Example 6.4]{MR1800714} and in \cite[Definition 2.1]{MR1667680}, respectively. In these  
papers also $\mathcal{H}_{\NA(V)}(t)$ was calculated. 
In both cases one can check that
\[
\partial_{v_4}\partial_{v_2}\partial_{v_4}\partial_{v_1}\partial_{v_2}\partial_{v_4}\partial_{v_3}\partial_{v_4}\partial_{v_2}\partial_{v_5}\partial_{v_6}\partial_{v_5}(v_1v_2v_1v_3v_4v_2v_1v_3v_4v_5v_1v_6)=-1.
\]
This completes the proof.
\end{proof}

\subsection{The rack $\mathcal{B}$}

Let $X=\mathcal{B}$. Then 
\begin{gather}
x_1^4=x_2^4=x_3^4=x_4^4=x_5^4=x_6^4,\\
x_3x_5=x_2x_4=x_1x_6
\end{gather}
in $G_X$.

\begin{lem}
The centralizer of $x_1$ in $G_X$ is the abelian subgroup generated by $x_1$
and $x_6$. These elements satisfy $x_1^4=x_6^4$.
\end{lem}

\begin{proof}
The proof is similar to the proof of Lemma \ref{lem:A_cen}.
\end{proof}

\begin{prop}
\label{pro:B_main}
Let $\rho$ be an absolutely simple representation of $C_{G_X}(x_1)$. Let $V=M(x_1,\rho)$ 
and $d_V=\dim V$.
\begin{enumerate}
\item The representation $\rho$ is a linear character on $C_{G_X}(x_1)$, hence $d_V=6$. 
Moreover, $\dim\NA_2(V)\leq d_V(d_V+1)/2$ if and only if $\rho(x_1)=\rho(x_6)=-1$. 
\item Assume that $\rho(x_1)=-1$ and $\rho(x_6)=-1$. Then the following hold.
\begin{enumerate}
\item $\mathcal{H}_{\NA(V)}(t)=(2)^2_t (3)^2_t (4)^2_t$ and $\dim\NA(V)=576$.
\item Let $v_i\in V_{x_i}$ with $i\in X$ be non-zero elements. Then the monomial 
$v_1v_2v_1v_3v_2v_1v_4v_3v_2v_5v_4v_6$ is an integral of $\NA(V)$.
\end{enumerate}
\end{enumerate}
\end{prop}

\begin{proof}
Analogous to the proof of Proposition \ref{pro:A_main} for $\fie=\ndC$.  The
Nichols algebra in part $(2)$ was studied in \cite[Theorem 6.12]{MR1994219}.
The existence of an integral of degree $12$ follows from the formula
\[
\partial_{v_2}\partial_{v_1}\partial_{v_4}\partial_{v_3}
\partial_{v_4}\partial_{v_6}\partial_{v_2}\partial_{v_6}
\partial_{v_4}\partial_{v_3}\partial_{v_4}\partial_{v_5}
(v_1v_2v_1v_3v_2v_1v_4v_3v_2v_5v_4v_6)=-1.
\]
For other fields see the proof of Proposition \ref{pro:T_main}. 
\end{proof}

\subsection{The rack $\mathcal{C}$}

Let $X=\mathcal{C}$. 

\begin{lem}
\label{lem:C_cen}
The centralizer of $x_1$ in $G_X$ is the non-abelian group 
$\langle x_1\rangle\times\langle x_8,x_9\rangle$. These generators
satisfy $x_1^2=x_8^2=x_9^2$ and $x_8x_9x_8=x_9x_8x_9$.
\end{lem}

\begin{proof}
The proof is similar to the proof of Lemma \ref{lem:A_cen}.
\end{proof}

\begin{prop}
\label{pro:C_main}
Let $\rho$ be an absolutely simple representation of $C_{G_X}(x_1)$. Let $V=M(x_1,\rho)$ 
and $d_V=\dim V$.
\begin{enumerate}
\item We have $\dim\NA_2(V)\leq d_V(d_V+1)/2$ if and only if $\rho$ is a linear character such that 
$\rho(x_1)=-1$ and $\rho(x_8)=\rho(x_9)=\pm1$. 
\item Assume that $\rho$ is a linear character such that 
$\rho(x_1)=-1$ and $\rho(x_8)=\rho(x_9)=\pm1$. Then the following hold.
\begin{enumerate}
\item $\mathcal{H}_{\NA(V)}(t)=(4)^4_t (5)^2_t (6)^4_t$ and $\dim\NA(V)=8294400$.
\item Let $v_i\in V_{x_i}$ with $i\in X$ be non-zero elements. Then the monomial 
\begin{equation}
\label{eq:integralC}
\begin{aligned}
&v_1v_2v_1v_3v_4v_1v_2v_1v_4v_5v_3v_6v_1v_2v_1v_4v_1v_2v_1v_4\times\\
&\quad\quad v_6v_1v_2v_1v_5v_3v_6v_2v_4v_2v_6v_7v_3v_5v_3v_7v_8v_9v_8v_{10}
\end{aligned}
\end{equation}
is an integral of $\NA(V)$.
\end{enumerate}
\end{enumerate}
\end{prop}

\begin{proof}
Assume that $\dim\NA_2(V)\leq d_V(d_V+1)/2$. Then $\rho$ is a linear character by
Theorem \ref{thm:YD}. Lemma \ref{lem:rho(g)=-1} yields that $\rho(x_1)=-1$. Equations 
$\rho(x_8)=\rho(x_9)=\pm1$ follow from Lemma \ref{lem:C_cen}. The remaining claim in $(1)$ 
follows from part $(2)$. 

The statement $(2)(a)$ was first given in \cite{MR1667680} in the case that $\rho(x_8)=1$ and 
in \cite{zoo} in the case that $\rho(x_8)=-1$. For a proof for $\fie=\ndC$ 
see for example \cite[Theorem 2.4]{GG2009}. For other fields proceed as in the proof of 
Proposition \ref{pro:T_main}. The evaluation of the product 
\begin{equation*}
\begin{aligned}
&\partial_{v_4} \partial_{v_1} \partial_{v_3} \partial_{v_5} \partial_{v_4} 
\partial_{v_2} \partial_{v_5} \partial_{v_4} \partial_{v_6} \partial_{v_7} 
\partial_{v_3} \partial_{v_9} \partial_{v_6} \partial_{v_{10}} \partial_{v_6} 
\partial_{v_5} \partial_{v_{10}} \partial_{v_6} \partial_{v_9}
\partial_{v_6}\times\\
&\quad\quad\partial_{v_7} \partial_{v_5} \partial_{v_6} \partial_{v_9} \partial_{v_8} 
\partial_{v_{10}} \partial_{v_7} \partial_{v_6} \partial_{v_5} \partial_{v_3} 
\partial_{v_5} \partial_{v_7} \partial_{v_9} \partial_{v_8} \partial_{v_{10}} 
\partial_{v_7} \partial_{v_{10}} \partial_{v_8} \partial_{v_{10}} \partial_{v_9}
\end{aligned}
\end{equation*}
of derivations at the monomial given in \eqref{eq:integralC} gives $\rho(x_8)$. 
Hence the monomial in \eqref{eq:integralC} is an integral.
\end{proof}

\subsection{Affine Racks}


Let $X$ be one of the affine racks listed in Table \ref{tab:racks}. 

\begin{lem}
The centralizer of $x_1$ in $G_X$ is the cyclic group
generated by $x_1$. 
\end{lem}

\begin{proof}
Assume that $X=\mathrm{Aff}(5,2)$. 
Then 
$\overline{G_X}$ is isomorphic to the affine group $C_5\rtimes C_4$. 
The centralizer in $\overline{G_X}$ of $\pi x_1$ in the cyclic subgroup generated 
by $\pi x_1$. Then the claim follows from Lemma \ref{lem:centralizer}.
The proof for the other affine racks is similar.
\end{proof}

\begin{prop}
\label{pro:aff_main}
Let $\rho$ be an absolutely simple representation of $C_{G_X}(x_1)$. Let $V=M(x_1,\rho)$ 
and $d_V=\dim V$.
\begin{enumerate}
\item The representation $\rho$ is a linear character on $C_{G_X}(x_1)$ and hence $d_V=\#X$. 
Moreover, $\dim\NA_2(V)\leq d_V(d_V+1)/2$ if and only if $\rho(x_1)=-1$. 
\item Assume that $X=\mathrm{Aff}(5,2)$ or $X=\mathrm{Aff}(5,3)$ and that $\rho(x_1)=-1$. Then the following hold.
\begin{enumerate}
\item $\mathcal{H}_{\NA(V)}(t)=(4)^4_t (5)_t$ and $\dim\NA(V)=1280$.
\item Let $v_i\in V_{x_i}$ with $i\in X$ be non-zero elements. Then 
\[
v_1v_2v_1v_2v_3v_1v_2v_1v_3v_1v_4v_1v_4v_2v_3v_5
\]
is an integral of $\NA(V)$ for $X=\mathrm{Aff}(5,2)$ and 
\[
v_1 v_2 v_4 v_3 v_2 v_4 v_5 v_2 v_1 v_3 v_2 v_4 v_3 v_1 v_2 v_4 
\]
is an integral of $\NA(V)$ for $X=\mathrm{Aff}(5,3)$. 
\end{enumerate}
\item Assume that $X=\mathrm{Aff}(7,3)$ or $X=\mathrm{Aff}(7,5)$ and that $\rho(x_1)=-1$. Then the following hold.
\begin{enumerate}
\item $\mathcal{H}_{\NA(V)}(t)=(6)^6_t (7)_t$ and $\dim\NA(V)=326592$.
\item Let $v_i\in V_{x_i}$ with $i\in X$ be non-zero elements. Then 
\begin{equation}
\label{eq:integral_aff73}
\begin{aligned}
&v_1 v_2 v_1 v_3 v_1 v_2 v_1 v_3 v_1 v_2 v_1 v_3 v_4 v_2 v_1 v_4 v_2 v_3\times\\
&\quad\quad v_4 v_2 v_1 v_5 v_1 v_3 v_1 v_2 v_1 v_3 v_4 v_2 v_3 v_5 v_1 v_6 v_4 v_7
\end{aligned}
\end{equation}
is an integral of $\NA(V)$ for $X=\mathrm{Aff}(7,3)$ and
\begin{equation}
\label{eq:integral_aff75}
\begin{aligned}
&v_6 v_7 v_6 v_5 v_6 v_7 v_5 v_6 v_5 v_7 v_6 v_5 v_4 v_5 v_6 v_4 v_5 v_7 \times\\
&\quad\quad v_6 v_5 v_7 v_3 v_7 v_6 v_2 v_3 v_4 v_2 v_4 v_3 v_5 v_4 v_3 v_2 v_1 v_2
\end{aligned}
\end{equation}
is an integral of $\NA(V)$ for $X=\mathrm{Aff}(7,5)$. 
\end{enumerate}
\end{enumerate}
\end{prop}

\begin{proof}
Assume that $\dim\NA_2(V)\leq d_V(d_V+1)/2$. Since $C_{G_X}(x_1)$ is abelian, the representation 
$\rho$ is a linear character. Lemma \ref{lem:rho(g)=-1} yields that $\rho(x_1)=-1$. 
The remaining claim in $(1)$ follows from parts $(2)$ and $(3)$. 

The statement $(2)$ was proved in \cite[Theorem 6.16]{MR1994219} in the case that $X=\mathrm{Aff}(5,2)$ and 
$\fie=\ndC$. In the proof of \cite[Theorem 6.16]{MR1994219} it was shown that
the evaluation of the product 
\[
\partial_{v_3} \partial_{v_1} \partial_{v_3} \partial_{v_1} \partial_{v_2} 
\partial_{v_1} \partial_{v_2} \partial_{v_3} \partial_{v_4} \partial_{v_3} 
\partial_{v_5} \partial_{v_3} \partial_{v_4} \partial_{v_5} \partial_{v_4} 
\partial_{v_5}
\]
of derivations at $v_1v_2v_1v_2v_3v_1v_2v_1v_3v_1v_4v_1v_4v_2v_3v_5$ gives $1$.
Thus for fields $\fie\ne\ndC$ the claim can be proved as in the proof of 
Proposition \ref{pro:T_main}. The case $X=\mathrm{Aff}(5,3)$ is similar. 
The evaluation of the product 
\[
\partial_{v_2} \partial_{v_1} \partial_{v_2} \partial_{v_1} \partial_{v_3} 
\partial_{v_1} \partial_{v_4} \partial_{v_3} \partial_{v_5} \partial_{v_3} 
\partial_{v_5} \partial_{v_4} \partial_{v_3} \partial_{v_5} \partial_{v_3} 
\partial_{v_5}
\]
of derivations at $v_1 v_2 v_4 v_3 v_2 v_4 v_5 v_2 v_1 v_3 v_2 v_4 v_3 v_1 v_2 v_4$ gives $1$.

The statement $(3)$ was proved in \cite{zoo} in the case that $\fie=\ndC$. 
For other fields proceed as in the proof of 
Proposition \ref{pro:T_main}. 
For $X=\mathrm{Aff}(7,3)$ the evaluation of the product 
\begin{equation}
\begin{aligned}
&\partial_{v_2} \partial_{v_1} \partial_{v_2} \partial_{v_3} \partial_{v_4} 
\partial_{v_5} \partial_{v_3} \partial_{v_4} \partial_{v_2} \partial_{v_4} 
\partial_{v_3} \partial_{v_2} \partial_{v_6} \partial_{v_7} \partial_{v_3} 
\partial_{v_7} \partial_{v_5} \partial_{v_6} \times\\
&\quad\quad\partial_{v_7} \partial_{v_5} \partial_{v_4} \partial_{v_6} \partial_{v_5} \partial_{v_4} \partial_{v_5} 
\partial_{v_6} \partial_{v_7} \partial_{v_5} \partial_{v_6} \partial_{v_5} 
\partial_{v_7} \partial_{v_6} \partial_{v_5} \partial_{v_6} \partial_{v_7} 
\partial_{v_6}
\end{aligned}
\end{equation}
of derivations at the monomial given in \eqref{eq:integral_aff73} gives $1$.

For $X=\mathrm{Aff}(7,5)$ the evaluation of the product 
\begin{equation*}
\begin{aligned}
&\partial_{v_6} \partial_{v_2} \partial_{v_7} \partial_{v_1} \partial_{v_2} \partial_{v_1} 
\partial_{v_3} \partial_{v_6} \partial_{v_5} \partial_{v_4} \partial_{v_1} \partial_{v_2} 
\partial_{v_1} \partial_{v_3} \partial_{v_1} \partial_{v_5} \partial_{v_1} \partial_{v_2} \times\\
&\quad\quad\partial_{v_4} \partial_{v_2} \partial_{v_6} \partial_{v_1} \partial_{v_3} \partial_{v_4} 
\partial_{v_3} \partial_{v_5} \partial_{v_7} \partial_{v_4} \partial_{v_3} \partial_{v_1} 
\partial_{v_4} \partial_{v_1} \partial_{v_3} \partial_{v_1} \partial_{v_2} \partial_{v_1}
\end{aligned}
\end{equation*}
of derivations at the monomial given in \eqref{eq:integral_aff75} gives $-1$.
\end{proof}

\begin{appendix}
\section{An application of the elementary divisors theorem}

The following lemma, which is an application of the elementary divisors theorem, is needed for
the proof of Theorem \ref{thm:diff_char}.

\begin{lem}
\label{lem:ed}
Let $V$ be a finite-dimensional vector space over $\ndQ$ and let $U\subseteq V$ be a subspace. 
Let $W$ be the free $\ndZ$-module generated by a basis of $V$ and let $U_\ndZ=U\cap W$.
Then for all primes $p$ the quotient 
$U_\ndZ/pU_\ndZ$ is a vector space over $\mathbb{F}_p$ of dimension $\dim U$.
\end{lem}

\begin{proof}
Let $d=\dim V$. By the elementary divisors theorem, see \cite[Ch. III, Theorem 7.8]{L}, there exist a basis 
$\{e_1,\dots,e_d\}$ of $W$, an integer $m\in\ndN$ and integers $a_1,\dots,a_m>0$ such that 
$\{a_1e_1,\dots,a_me_m\}$ is a basis of $U_\ndZ$. Since $U$ is a vector space, it follows that 
$a_i=1$ for all $i\in\{1,\dots,m\}$. We conclude that $\{e_1,\dots,e_m\}$ is a basis of $U$ by 
definition of $U_\ndZ$. Thus $\dim(U_\ndZ/pU_\ndZ)=\dim U=m$.
\end{proof}

\end{appendix}

\begin{acknowledgement*}
I.H. was supported by the German Research Foundation via a Heisenberg fellowship. 
M.G. and L.V. was supported by CONICET. 
The third author thanks DAAD and Philipps-Universit\"at Marburg for the support 
of his visit in Marburg from December 2009 to March 2010.
\end{acknowledgement*}


\def\cprime{$'$}

\end{document}